\documentclass[10pt, a4paper, reqno, oneside]{amsart}

\usepackage{amsmath, amsfonts, amsthm, enumerate, mathtools, multicol, scalefnt, relsize}

\usepackage[mathscr]{euscript}
\usepackage{mathbbol}
\usepackage[latin1]{inputenc}
\usepackage{graphicx}
\usepackage[all]{xy}
\usepackage{enumitem}
\usepackage{autobreak}
\setlist[itemize]{noitemsep, topsep=1pt, leftmargin=20pt}
\usepackage{xfrac}
\usepackage{todonotes}

\usepackage[cm]{fullpage}

\usepackage{listings}
\usepackage[hypertexnames=false,
backref=page,
    pdfpagemode=UseNone,
    breaklinks=true,
    extension=pdf,
    colorlinks=true,
    linkcolor=blue,
    citecolor=blue,
    urlcolor=blue,
]{hyperref}

\newcommand\bcdot{\ensuremath{
  \mathchoice
   {\mskip\thinmuskip\lower0.2ex\hbox{\scalebox{1.6}{$\cdot$}}\mskip\thinmuskip}}
   {\mskip\thinmuskip\lower0.2ex\hbox{\scalebox{1.6}{$\cdot$}}\mskip\thinmuskip}
   {\lower0.3ex\hbox{\scalebox{1.2}{$\cdot$}}}
   {\lower0.3ex\hbox{\scalebox{1.2}{$\cdot$}}}
}
\theoremstyle{plain}
\newtheorem{theo}{Theorem}[section]

\newtheorem{prop}[theo]{Proposition}

\theoremstyle{definition}
\newtheorem{rem}[theo]{Remark}

\newtheorem{definition}[theo]{Definition}

\theoremstyle{plain}
\newtheorem{lemma}[theo]{Lemma}
\newtheorem{theorem}[theo]{Theorem}
\newtheorem{corollary}[theo]{Corollary}
\newtheorem{proposition}[theo]{Proposition}

\theoremstyle{definition}

\newtheorem{remark}[theo]{Remark}

\theoremstyle{plain}

\renewcommand{\=}{\coloneqq}

\renewcommand{\a}{\alpha}
\renewcommand{\b}{\beta}

\renewcommand{\d}{\delta}
\newcommand{\e}{\varepsilon}
\newcommand{\f}{\varphi}

\newcommand{\w}{\omega}
\newcommand{\q}{\vartheta}

\newcommand{\G}{\Gamma}

\renewcommand{\L}{\Lambda}

\newcommand{\W}{\Omega}

\DeclareSymbolFontAlphabet{\mathbb}{AMSb}
\DeclareSymbolFontAlphabet{\mathbbl}{bbold}

\newcommand{\bC}{\mathbb{C}}

\newcommand{\bR}{\mathbb{R}}
\newcommand{\bZ}{\mathbb{Z}}

\newcommand{\bN}{\mathbb{N}}

\newcommand{\fG}{\mathsf{G}}
\newcommand{\fH}{\mathsf{H}}

\newcommand{\fO}{\mathsf{O}}

\newcommand{\fU}{\mathsf{U}}

\newcommand{\fGL}{\mathsf{GL}}
\newcommand{\fSL}{\mathsf{SL}}

\renewcommand{\gg}{\mathfrak{g}}
\newcommand{\gh}{\mathfrak{h}}

\newcommand{\gj}{\mathfrak{j}}

\newcommand{\gl}{\mathfrak{l}}
\newcommand{\gm}{\mathfrak{m}}

\newcommand{\go}{\mathfrak{o}}

\newcommand{\gs}{\mathfrak{s}}

\newcommand{\gu}{\mathfrak{u}}

\newcommand{\gS}{\mathfrak{S}}

\newcommand{\so}{\mathfrak{so}}

\newcommand{\gkill}{\mathfrak{kill}}

\newcommand\GL{\mathrm{GL}}

\newcommand\Sym{\mathrm{Sym}}
\newcommand\Skew{\mathrm{Skew}}

\newcommand{\cC}{\mathcal{C}}

\newcommand{\cW}{\mathcal{W}}
\newcommand{\cX}{\mathcal{X}}

\newcommand{\eA}{\EuScript{A}}
\newcommand{\eB}{\EuScript{B}}

\newcommand{\eH}{\EuScript{H}}

\newcommand{\eL}{\EuScript{L}}

\newcommand{\eP}{\EuScript{P}}

\newcommand{\eU}{\EuScript{U}}

\DeclareMathOperator\tr{tr}
\DeclareMathOperator\Tr{Tr}
\DeclareMathOperator\Lie{Lie}
\DeclareMathOperator\End{End}

\DeclareMathOperator\scal{scal} 

\DeclareMathOperator\ad{ad}

\DeclareMathOperator\Exp{Exp}
\DeclareMathOperator\diff{d\!}

\DeclareMathOperator\inj{inj}

\DeclareMathOperator\vspan{span}
\DeclareMathOperator\st{st}
\newcommand{\p}{\partial}
\newcommand{\ti}{\mathtt{i}}
\newcommand{\td}{\mathtt{d}}

\newcommand{\Rm}{\operatorname{Rm}}
\newcommand{\ol}{\overline}
\newcommand{\wt}{\widetilde}
\newcommand{\zero}{\operatorname{o}}

\newcommand{\n}{\nabla}

\makeatletter
\@namedef{subjclassname@2020}{\textup{2020} Mathematics Subject Classification}
\makeatother

\allowdisplaybreaks[3]

\title[]{A survey on locally Homogeneous almost-Hermitian spaces}

\author{Daniele Angella}
\address[Daniele Angella]{Dipartimento di Matematica e Informatica ``Ulisse Dini''\\
Universit\`a degli Studi di Firenze,
viale Morgagni 67/a,
	50134 Firenze, Italy}
\email{daniele.angella@unifi.it}
\email{daniele.angella@gmail.com}

\author{Francesco Pediconi}
\address[Francesco Pediconi]{Dipartimento di Matematica e Informatica ``Ulisse Dini''\\
	Universit\`a degli Studi di Firenze,
	viale Morgagni 67/a,
	50134 Firenze, Italy}
\email{francesco.pediconi@unifi.it}

\thanks{
The authors are supported by project PRIN2017 ``Real and Complex Manifolds: Topology, Geometry and holomorphic dynamics'' (code 2017JZ2SW5), and by GNSAGA of INdAM
} 

\begin{document}

\begin{abstract}
We survey the theory of locally homogeneous almost-Hermitian spaces. In particular,  by using the framework of varying Lie brackets, we write formulas for the curvature of all the Gauduchon connections and we provide explicit examples of computations.
\end{abstract}

\maketitle

\section{Introduction}

In Differential Geometry, the notions of symmetries and local symmetries arise naturally and play a central role in many geometric problems.
The geometry of locally homogeneous Riemannian spaces $(M,g)$ is well understood, starting from the foundational paper by Nomizu \cite{nomizu} on local Killing vector fields, proceeding with the work by Palais, Tricerri, and many others; we refer e.g. to \cite{palais, tricerri-vanhecke, tricerri}. (See \cite{pediconi-thesis} and the references therein for an up-to-date account.)
More precisely, their local geometry is encoded in the Lie algebra $\gg$ of {\em Killing generators}, that are the pairs $(v,A) \in T_pM \oplus \so(T_pM,g_p)$ such that
\begin{equation*}
v \,\lrcorner\, \big((D^g)^{k+1}\Rm(g)\big){}_p + A \cdot \big((D^g)^k\Rm(g)\big){}_p = 0 \quad \text{ for any $k \in \bZ_{\geq0}$} \,\, ,
\end{equation*}
where $p\in M$ is a point, $\so(T_pM,g_p)$ acts on the tensor algebra over $T_pM$ as a derivation, $D^g$ denotes the Levi-Civita connection and $\Rm(g)$ is the Riemannian curvature tensor.
Indeed, by the condition of locally homogeneity, the vectors $v$, varying $(v,A)\in\gg$, span $T_pM$.
Moreover, since locally homogeneous spaces are real analytic Riemannian manifolds (see e.g. \cite[Lemma 1.1]{bohm-lafuente-simon} for a modern proof), by \cite{nomizu} and \cite{palais}, there exists a neighborhood of $p$ which is locally isometric to the {\em local quotient space} $\fG/\fH$, where $\fG$ is the simply-connected Lie group with Lie algebra $\gg$, and $\fH$ is the (possibly non-closed) connected Lie subgroup of $\fG$ with Lie algebra $\gh\=\{(0,A) \in \gg\}$.
Finally, two local quotient spaces are locally equivariantly isometric if and only if the corresponding algebras of Killing generators are isomorphic in the category of the so-called {\em orthogonal transitive Lie algebras}, see e.g. \cite[Section 2]{pediconi-geomded}.

By definition, an orthogonal transitive Lie algebra is the algebraic datum of $(\gg=\gh+\gm,\langle\,,\rangle)$, where $\gg$ is a Lie algebra, $\gh \subset \gg$ is a Lie subalgebra that does not contain any non-trivial ideal of $\gg$, $\gm$ is an $\ad(\gh)$-invariant complement of $\gh$ in $\gg$, and $\langle\,,\rangle$ is an $\ad(\gh)$-invariant Euclidean product on $\gm$.
If we denote $m\=\dim\gm$ and $q\=\dim\gh$, these data are encoded by equivalence classes of tensors
$$\mu \in \big(\fGL(q) \times \fO(m)\big) \big\backslash \big(\L^2(\bR^{q+m})^*\otimes \bR^{q+m}\big)$$
satisfying appropriate conditions (compare with Definition \ref{utLa}), called {\em abstract brackets}.
Following Lauret, see \cite{lauret-jlms}, one can use these abstract brackets in order to parametrize the moduli space of locally homogeneous spaces, up to local equivariant isometries.

This approach of varying Lie brackets, rather than metrics, provides a convenient setting for variational problems \cite{milnor, lauret-jlms, lauret-mathz, lauret-tams, lauret-rend, lafuente, arroyo-lafuente}.
Moreover, we stress that locally homogeneous Riemannian spaces provide a natural completion for homogeneous Riemannian spaces with respect to various notions of convergence, see e.g. \cite{pediconi-annsns, pediconi-geomded} and so this appears as the right framework in order to study geometric evolution equations \cite{lott, bohm-lafuente}. Moreover, it happens that one needs to study this completion in order to get results on homogeneous spaces, see e.g. \cite[Theorem 4]{bohm-lafuente-simon}.

\smallskip

In this note, we translate the above theory to the almost-Hermitian setting. We consider {\em locally homogeneous almost-Hermitian spaces}, namely, almost-Hermitian manifolds $(M,J,g)$ such that the pseudogroup of local pseudo-holomorphic isometries acts transitively. In Lemma \ref{lemma:analyticity}, we show that we can work in the real-analytic category without loss of generality.
By considering the infinitesimal action of Killing vector fields on the almost-complex structure, see Lemma \ref{lemma:DkLJ}, we are lead to introduce the Lie algebra of {\em real holomorphic Killing generators} at a point $p\in M$ as the set of pairs $(v,A) \in T_pM \oplus \so(T_pM,g_p)$ such that
\begin{equation*}
\begin{gathered}
v \,\lrcorner\, \big((D^g)^{k+1}J\big){}_p + A \cdot \big((D^g)^kJ\big){}_p = 0 \,\, , \quad v \,\lrcorner\, \big((D^g)^{k+1}\Rm(g)\big){}_p + A \cdot \big((D^g)^k\Rm(g)\big){}_p = 0 \quad \text{ for any $k \in \bZ_{\geq0}$} \,\, .
\end {gathered}
\end{equation*}
They correspond to infinitesimal real holomorphic Killing vector fields, that is, vector fields such that the local flow is made by pseudo-holomorphic isometric local transformations.
Following the same approach as before, one can show that these data at the point $p$ encode the local geometry of $(M,J,g)$.

The moduli space of locally homogeneous almost-Hermitian spaces is then parametrized by using {\em unitary transitive Lie algebras}, namely $(\gg=\gh+\gm,I,\langle\,,\rangle)$, where $(\gg=\gh+\gm,\langle\,,\rangle)$ is as before, and $I$ is an $\ad(\gh)$-invariant linear complex structure on $\gm$ that satisfies $\langle I{\cdot},I{\cdot\cdot}\rangle=\langle{\cdot},{\cdot\cdot}\rangle$, see Theorem \ref{thm:Sp}.
Again, these algebraic data are encoded by equivalence classes of tensors
$$ \mu \in \big(\fGL(q) \times \fU(m)\big) \big\backslash \big(\L^2(\bR^{q+2m})^*\otimes \bR^{q+2m}\big) $$
as before, satisfying a further compatibility condition with respect to the linear complex structure of $\bC^m=\bR^{2m}$.

\smallskip

After the foundational work by Gauduchon \cite{gauduchon-bumi}, any almost-Hermitian manifold is endowed with a distinguished one-parameter family of Hermitian connections, that are called {\em Gauduchon connections}. They include, among others, the {\em Chern connection} and the {\em Bismut connection}, which are fundamental tool to investigate the (almost) complex geometry of the manifold. Notice that, in the K\"ahler case, they all coincide with the Levi-Civita connection, which in the general non-K\"ahler setting is not even adapted to the complex structure.
Remarkably, by restricting to locally homogeneous almost-Hermitian spaces, all the geometric quantities related to the Gauduchon family can be expressed in purely algebraic terms depending on $\mu$, see Section \ref{Curvmu}.

As explicit examples, we apply this approach to compute the Gauduchon curvatures of locally homogeneous (almost-)Hermitian structures on the Iwasawa threefold, on the primary Kodaira surface, and on the almost-K\"ahler Kodaira-Thurston four-manifold. We make use of the symbolic computation software SageMath \cite{sagemath}.

\smallskip

As in the Riemannian case, local symmetries could be useful to understand special Hermitian metrics (see e.g. \cite{belgun, fino-parton-salamon, ugarte-villacampa, angella-calamai-spotti-2, podesta, angella-pediconi-1}) and variational problems in Hermitian and almost-Hermitian geometry, in particular, geometric flows driven by Hermitian curvatures (see e.g. \cite{tosatti-weinkove, boling, ustinovskij, enrietti-fino-vezzoni, panelli-podesta, lafuente-pujia-vezzoni, pediconi-pujia, angella-pediconi-2, arroyo-lafuente}) including convergence notions (see Section \ref{sec:potpourri}).

\medskip

The paper is organized as follows.
In Section \ref{prel}, we recall some preliminary notions and notation on complex linear algebra and Hermitian geometry.
In Section \ref{sec:locally-homog}, we introduce locally homogeneous almost-Hermitian spaces and their Hermitian Nomizu algebras.
In Section \ref{sec:geom-models}, we introduce the notion of almost-Hermitian geometric models and we show their compactness with respect to the Cheeger-Gromov convergence.
In Section \ref{sec:space-locally-homog}, we give a treatment of the moduli space of locally homogeneous almost-Hermitian spaces. We also write explicit formulas for the curvatures of Gauduchon connections.
In Section \ref{sec:examples}, we investigate explicit examples of locally homogeneous (almost) Hermitian metrics on the Iwasawa manifold, on the primary Kodaira surface, and on the almost-complex Kodaira-Thurston manifold.
Finally, in Appendix \ref{app:sagemath}, we collect the SageMath code for the previous examples.

\bigskip

\noindent {\itshape Acknowledgements.}
This note has been written for the special volume collecting the Proceedings of the meeting ``Cohomology of Complex Manifolds and Special Structures, II'' that was held in Levico Terme on July 04-09, 2021. The authors are grateful to the Organizers of the meeting, Costantino Medori, Massimiliano Pontecorvo, Adriano Tomassini, for the kind invitation and the fruitful atmosphere in Levico, and to CIRM-FBK for the support.

\medskip
\section{Preliminaries and notation} \label{prel}
\setcounter{equation} 0

We denote by $I_{\st}$, $\langle\,,\rangle_{\st}$ the standard linear complex structure and the standard Euclidean inner product on $\bR^{2m}$, respectively, that are defined by $I_{\st}e_{2i-1}^{\zero}=e_{2i}^{\zero}$ and $\langle e_i^{\zero},e_j^{\zero}\rangle_{\st} = \d_{ij}$ with respect to the standard basis $(e_1^{\zero},{\dots},e_{2m}^{\zero})$ of $\bR^{2m}$. We will also denote by $B_{\st}(x,r)$ the standard Euclidean ball in $\bR^{2m}$ centered at $x \in \bR^{2m}$ with radius $r>0$. Any integrable almost-complex structure will be just called {\it complex structure}. We will use the word {\it smooth} as a synonym for {\it of class $\cC^{\infty}$}.

\subsection{Complex linear algebra} \hfill \par

Let $V=(V,J,g)$ be a real vector space of even dimension $\dim_{\bR}V=2m$ endowed with a linear complex structure $J$ and an Euclidean scalar product $g$ such that $g(J(\cdot),J(\cdot\cdot))=g(\cdot,\cdot\cdot)$. Fix a $(J,g)$-unitary basis $(e_i,Je_i)$ for $V$ and consider the associated complex basis $\big(\e_i \= \tfrac1{\sqrt2}(e_i-\ti Je_i) , \, \e_{\bar{i}} \= \tfrac1{\sqrt2}(e_i+\ti Je_i)\big)
$
for the complexification $V^{\bC}\=V \otimes_{\bR} \bC$, which splits as a sum of $J$-eigenspaces $V^{\bC} = V^{1,0} \oplus V^{0,1}$. Here, we use the fact that any real tensor on $V$ can $\bC$-linearly extended to $V^{\bC}$ in a unique way. One can directly check that $J\e_i=\ti \e_i$, $J\e_{\bar{i}}=-\ti \e_{\bar{i}}$ and $\overline{\e_i}=\e_{\bar{i}}$. Notice that $J$ acts on covectors $\q \in V^*$ via $(J\q) \= \q \circ J^{-1}$, so that $(e^i,Je^i)$ is the dual basis of $(e_i,Je_i)$ for $V^*$. Analogously, it holds that
$\big(
\e^i \= \tfrac1{\sqrt2}(e^i+\ti Je^i) , \, \e^{\bar{i}} \= \tfrac1{\sqrt2}(e_i-\ti Je_i)\big)
$
is the dual basis of $(\e_i,\e_{\bar{i}})$, and $J\e^i=-\ti \e^i$, $J\e^{\bar{i}}=\ti \e^{\bar{i}}$. With respect to such basis, we have
$$
g = \d_{\bar{j}i} \, \e^i \odot \e^{\bar{j}} \,\, , \quad \text{ with } \,\, \e^i \odot \e^{\bar{j}} \= \e^i \otimes \e^{\bar{j}} + \e^{\bar{j}} \otimes \e^i \,\, .
$$
We consider now the spaces
$$\begin{aligned}
\Sym^{1,1}(V) &\= \{h \in \End(V) : g(h(\cdot),\cdot\cdot)=g(\cdot,h(\cdot\cdot)) \, , \,\, [h,J]=0 \} \,\, , \\
\Skew^{1,1}(V) &\= \{\tilde{h} \in \End(V) : g(\tilde{h}(\cdot),\cdot\cdot)=-g(\cdot,\tilde{h}(\cdot\cdot)) \, , \,\, [\tilde{h},J]=0 \} \,\, ,
\end{aligned}$$
and we observe that the linear map
$$
\Phi_J : \Sym^{1,1}(V) \to \Skew^{1,1}(V) \,\, , \quad h \mapsto \tilde{h} = J \circ h
$$
is an isomorphism, with inverse given by $\tilde{h} \mapsto h= -J \circ \tilde{h}$. We denote by $\tr: \Sym(V) \to \bR$ the trace of symmetric endomorphisms $h: V \to V$ and, on the subspace $\Sym^{1,1}(V)$, we set $\tr^{\bC}: \Sym^{1,1}(V) \to \bR$ to be the trace of the induced complex endomorphism $h: V^{1,0} \to V^{1,0}$. Notice that $\tr(h) = 2\tr^{\bC}(h)$ for any $h \in \Sym^{1,1}(V)$.
Finally, we consider the space of real $(1,1)$-forms
$$
\L^{1,1}(V^*) \= \{ \a \in \L^2(V^*) : \a(J(\cdot),J(\cdot\cdot))=\a(\cdot,\cdot\cdot) \}
$$
and the projection
$
\pi^{1,1} : \L^{2}(V^*) \to \L^{1,1}(V^*)$ given by $(\pi^{1,1}\a)({\cdot},{\cdot\cdot}) \= \tfrac12\big(\a({\cdot},{\cdot\cdot})+\a(J{\cdot},J{\cdot\cdot})\big)
$.
Then, we observe that the linear map
\begin{equation}\label{eq:isom-symm}
\varsigma_g: \Sym^{1,1}(V) \to \L^{1,1}(V^*) \,\, , \quad \varsigma_g(h) \= g((J \circ h) \,\cdot,\cdot\cdot)
\end{equation}
is an isomorphism and so, accordingly, we define
$$
\Tr^{\bC}_g : \L^{1,1}(V^*) \to \bR \,\, , \quad \Tr^{\bC}_g(\a) \= \tr^{\bC}\!\big(\varsigma_g^{-1}(\a)\big) \,\, .
$$
Since any $\a \in \L^{1,1}(V^*)$ is of the form
$
\a = \a_{\bar{j}i} \,\ti\, \e^i \wedge \e^{\bar{j}}$,  with $\e^i \wedge \e^{\bar{j}} \= \e^i \otimes \e^{\bar{j}} - \e^{\bar{j}} \otimes \e^i $,
an easy computation shows that
$
\Tr^{\bC}_g(\a) =
\d^{i\bar{j}}\a_{\bar{j}i}$.
For the sake of notation, we set $\Tr^{\bC}_g(\a) \= \Tr^{\bC}_g(\pi^{1,1}\a)$ for any $\a \in \L^{2}(V^*)$.

Lastly, it will be useful to consider the decomposition of real $3$-forms
\begin{equation} \label{eq:decomposition}
\L^3V^*=\L^3_{+}V^* \oplus \L^3_{-}V^*
\end{equation}
given by the subspaces
$$
\L^3_{+}V^* \= \big( (\L^{2}(V^*)^{1,0}\otimes (V^*)^{0,1})\oplus((V^*)^{1,0}\otimes\L^{2}(V^{*})^{0,1}) \big) \cap \L^3 V^* \,\, , \quad 
\L^3_{-}V^* \= \big( \L^{3}(V^{*})^{1,0}\oplus\L^{3}(V^{*})^{0,1} \big) \cap \L^3 V^* \,\,.
$$
According to \eqref{eq:decomposition}, for any $\f\in\L^3V^*$, we will write $\f=\f^++\f^-$.

\subsection{Basics on Hermitian geometry} \hfill \par

Let $(M^{2m},J,g)$ be an almost-Hermitian manifold, i.e. $M$ is a smooth manifold of $\dim_{\bR}M=2m$ endowed with an almost-complex structure $J$ and a Riemannian metric $g$ such that $g(J{\cdot},J{\cdot\cdot})=g({\cdot},{\cdot\cdot})$.
We recall that a smooth map whose differential preserves $J$ (resp. $g$) is said to be {\it pseudo-holomorphic} (resp. {\it isometric}).
We denote by $\w \= g(J{\cdot},\cdot\cdot)$ its fundamental $2$-form, by $D^g$ the Levi-Civita connection of $g$, by $\Rm(g)(X,Y) \= D^g_{[X,Y]}-[D^g_X,D^g_Y]$ its Riemannian curvature operator and by $\sec(g)(X,Y) \= g(\Rm(g)(X,Y)X,Y)$ its sectional curvature. For any point $x \in M$, we denote by $\Exp(g)_x$ the Riemannian exponential at $x$, by $\inj_x(M,g)$ the injectivity radius at $x$ of the underlying Riemannian manifold and by $\eB_g(x,r)$ the Riemannian distance ball in $M$ centered at $x$ with radius $r$.

We set
$$
N_J(X,Y) \= [JX,JY]-[X,Y]-J([JX,Y]+[X,JY])
$$
to be the {\it Nijenhuis tensor of $J$}. By the foundational result of Newlander-Nirenberg, $J$ is integrable if and only if $N_J=0$.
We also set $\diff^{\,c} \= J^{-1} \! \circ \diff\, \circ J$. In particular, in the integrable setting, it holds
$$
\diff = \p+\bar{\p} \,\, , \quad
\diff^{\,c} = -\ti(\p-\bar{\p}) \,\, , \quad
\diff\diff^{\,c}=2\ti\p\bar{\p} \,\,.
$$

\begin{remark} \label{rem:real-analytic}
Let us recall that $J$ is integrable if $M$ admits a real-analytic structure $\eA_J =\{(\eU_{\a},\xi_{\a})\}$, compatible with its smooth structure, such that
$$
\diff \xi_{\a}(x) \circ J_x \circ \diff \xi_{\a}(x)^{-1} = I_{\st} \quad \text{ for any $\a$, for any $x \in \eU_{\a}$} \, .
$$
Notice that $J$ turns out to be a real-analytic tensor field, while $g$ in general is just smooth. However, if $g$ is real-analytic with respect to some real-analytic structure $\eA'$ on $M$, then one can assume that $\eA'=\eA_J$. Indeed, all the real-analytic structures on a given smooth manifold are equivalent up to a real-analytic diffeomorphism, see e.g. \cite{morrey, grauert}.
Also, in the possibly non-integrable setting, we will see in Lemma \ref{lemma:analyticity} that locally homogeneous almost-Hermitian structures are real-analytic.
\end{remark}

A linear connection $\n$ on $(M,J,g)$ is said to be {\it Hermitian} if leaves both $J,g$ parallel, i.e. $\n J = \n g =0$. Among such connections there are the so called {\it $t$-Gauduchon connections}, named after \cite{gauduchon-bumi}, that are defined by
\begin{multline} \label{eq:gauduchon-almost}
g(\n^t_XY,Z) \= g(D^g_XY,Z) -\tfrac{t+1}4(\diff^{\,c}\w)^+(X,JY,JZ) -\tfrac{t-1}4(\diff^{\,c}\w)^+(X,Y,Z) \\
-\tfrac14 g(X,N_J(Y,Z)) -\tfrac12(\diff^{\,c}\w)^-(X,Y,Z) \,\, ,
\end{multline}
for any $t \in \bR$.
In particular, when $J$ is integrable, which is equivalent to $N_J=0$, then
$$
(\diff^{\,c}\w)^+=\diff^{\,c}\w \,\, , \quad (\diff^{\,c}\w)^-=0
$$
and therefore \eqref{eq:gauduchon-almost} reduces to
\begin{equation} \label{eq:gauduchon}
g(\n^t_XY,Z) = g(D^g_XY,Z) -\tfrac{t+1}4\diff\w(JX,Y,Z) -\tfrac{t-1}4\diff\w(JX,JY,JZ) \,\, .
\end{equation}
The $t$-Gauduchon connections are Hermitian and their {\it torsion} $T^t=T^t(J,g)$, that is defined by $T^t(X,Y) \= \n^t_XY -\n^t_YX -[X,Y]$, is given by
\begin{multline} \label{eq:tors}
g(T^t(X,Y),Z) = -\tfrac{t+1}4 (\diff^{\,c}\w)^+(X,JY,JZ)-\tfrac{t+1}4(\diff^{\,c}\w)^+(JX,Y,JZ) -\tfrac{t-1}2(\diff^{\,c}\w)^+(X,Y,Z) \\
- (\diff^{\,c}\w)^-(X,Y,Z) -\tfrac14 g(N_J(Y,Z),X)+\tfrac14 g(N_J(X,Z),Y) \,\, .
\end{multline}
In the case when $J$ is integrable, the previous formula reduces to
$$
g(T^t(X,Y),Z) = -\tfrac{t+1}4 \diff\w(JX,Y,Z) -\tfrac{t+1}4 \diff\w(X,JY,Z) -\tfrac{t-1}2\diff\w(JX,JY,JZ) \,\, .
$$
Notice that there exists a 1-form $\q = \q(J,g)$ satisfying
$$
\tr(T^t(X,\cdot)) = \tfrac{t+1}2 \q(X) \,\, , \quad \diff \w^{m-1} = \q \wedge \w^{m-1}
$$
which is called {\it Lee form}, see \cite[Eqn (2.5.11)]{gauduchon-bumi}, \cite[Eqns (13) and (16)]{gauduchon-mathann}. We also define the {\it $t$-Gauduchon curvature operator} $\W^t=\W^t(J,g)$ by
\begin{equation} \label{Chern-curv}
\W^t \in \cC^{\infty}(M,\L^{2}(T^*M) \otimes \Skew_g^{1,1}(TM)) \,\, , \quad \W^t(X,Y) \= \n^t_{[X,Y]} - [\n^t_X,\n^t_Y] \,\, .
\end{equation}
Moreover, we call {\it first $t$-Gauduchon-Ricci $(1,1)$-form} the tensor field
\begin{equation}\label{eq:rho-1}
\begin{gathered}
\rho^{t,(1)} = \rho^{t,(1)}(J,g) \in \cC^{\infty}(M,\L^{1,1}(T^*M)) \,\, , \\
\rho^{t,(1)}(X,Y) \= \tr^\bC (\Phi_J^{-1}(\pi^{1,1}\Omega^t)(X,Y)) = -\tfrac12\tr^{\bC}(J \circ \W^t(X,Y)+J \circ \W^t(JX,JY)) \,\, .
\end{gathered}
\end{equation}
Notice that, according to our notation, $\rho^{t,(1)}$ denotes the $(1,1)$-component of the $2$-form obtained by tracing the curvature $\Omega^t$ with respect to the endomorphism part.
Analogously, we call {\it second $t$-Gauduchon-Ricci form} the tensor field
\begin{equation}\label{eq:rho-2}
\rho^{t,(2)} = \rho^{t,(2)}(J,g) \in \cC^{\infty}(M,\L^{1,1}(T^*M)) \,\, , \quad \rho^{t,(2)}(X,Y) \= g(\Tr^{\bC}_g(\W^t({\cdot},{\cdot\cdot}))X,Y)
\end{equation}
and {\it $t$-Gauduchon scalar curvature} the trace
\begin{equation}\label{eq:def-scal}
\scal^{t} = \scal^{t}(J,g) \in \cC^{\infty}(M,\bR) \,\, , \quad \scal^{t} \= 2\Tr^{\bC}_g\!\big(\rho^{t,(1)}\big) = 2\Tr^{\bC}_g\!\big(\rho^{t,(2)}\big) \,\, .
\end{equation}
Notice that the isomorphism \eqref{eq:isom-symm} allows to consider also the symmetric Ricci endomorphisms associated to \eqref{eq:rho-1} and \eqref{eq:rho-2}, see e.g. \cite[Sect 2.2]{angella-pediconi-1}.

\subsection{A comparison between the Gauduchon connections and the Levi-Civita connection} \hfill \par

Let $(M,J,g)$ be an almost-Hermitian Riemannian manifold and consider a metric linear connection $\n$ on $M$. Let us denote by $\W^{\nabla}$ its curvature and by $T^{\n}$ its torsion. Then, the difference $\G^{\n} \= \n - D^g$ is a $(1,2)$-tensor field, which is related to the torsion $T^{\n}$ by the following 

\begin{lemma}
The tensor fields $T^{\n}$ and $\G^{\n}$ verify the following equations:
\begin{equation} \label{eq:tors-cont} \begin{gathered}
T^{\n}(X,Y) = \G^{\n}(X,Y)-\G^{\n}(Y,X) \,\, , \\
2g\big(\G^{\n}(X,Y),Z\big) = g\big(T^{\n}(X,Y),Z\big) -g\big(T^{\n}(Y,Z),X\big) +g\big(T^{\n}(Z,X),Y\big) \,\, .
\end{gathered} \end{equation}
\end{lemma}

\begin{proof}
Firstly, since $D^g$ is torsion free, we get
$$\begin{aligned}
T^{\n}(X,Y) &= \n_XY -\n_YX -[X,Y] \\
&= D^g_XY -D^g_YX -[X,Y] +\G^{\n}(X,Y) -\G^{\n}(Y,X) \\
&= \G^{\n}(X,Y) -\G^{\n}(Y,X) \,\, .
\end{aligned}$$
Secondly, since $\n g = D^g g =0$, we obtain
$$\begin{aligned}
0 &= (\n_Xg)(Y,Z) +(\n_Yg)(X,Z) -(\n_Zg)(X,Y) \\
&= (D^g_Xg)(Y,Z) -g\big(\G^{\n}(X,Y),Z\big) -g\big(Y,\G^{\n}(X,Z)\big) +(D^gYg)(Z,X) -g\big(\G^{\n}(Y,Z),X\big) \\
&\quad -g\big(Z,\G^{\n}(Y,X)\big) -(D^g_Zg)(X,Y) +g\big(\G^{\n}(Z,X),Y\big) +g\big(X,\G^{\n}(Z,Y)\big) \\
&= -g\big(\G^{\n}(X,Y),Z\big) -g\big(\G^{\n}(Y,X),Z\big) -g\big(\G^{\n}(Y,Z)-\G^{\n}(Z,Y),X\big) +g\big(\G^{\n}(Z,X)-\G^{\n}(X,Z),Y\big) \\
&= -2g\big(\G^{\n}(X,Y),Z\big) +g\big(T^{\n}(X,Y),Z\big) -g\big(T^{\n}(Y,Z),X\big) +g\big(T^{\n}(Z,X),Y\big)
\end{aligned}$$
that completes the proof.
\end{proof}

We also remark that the curvatures $\W^{\n}$ and $\Rm(g)$ are related by the following

\begin{lemma}
The difference $\W^{\n}-\Rm(g)$ is explicitly given by
\begin{equation} \label{eq:diffcurvs}
\W^{\n}(X,Y) -\Rm(g)(X,Y) = X\lrcorner(D^g_Y\G^{\n}) -Y\lrcorner(D^g_X\G^{\n})- \big[\G^{\n}_X, \G^{\n}_Y\big] \,\, ,
\end{equation}
where $\G^{\n}_X \= \G^{\n}(X,\cdot)$ and $X\lrcorner(D^g_Y\G^{\n}) = (D^g_Y\G^{\n})(X,\cdot)$.
\end{lemma}

\begin{proof}
By the very definition
$$\begin{aligned}
\W^{\n}(X,Y) &= \n_{[X,Y]} -[\n_X, \n_Y] \\
&= D^g_{[X,Y]} +\G^{\n}_{[X,Y]} -\big[D^g_X +\G^{\n}_X, D^g_Y +\G^{\n}_Y\big] \\
&= \Rm(g)(X,Y) -\big[D^g_X, \G^{\n}_Y\big] +\big[D^g_Y,\G^{\n}_X\big] +\G^{\n}_{[X,Y]} -\big[\G^{\n}_X, \G^{\n}_Y\big] \,\, .
\end{aligned}$$
Since
$$\begin{aligned}
\big(-\big[D^g_X, \G^{\n}_Y\big] &+\big[D^g_Y,\G^{\n}_X\big] +\G^{\n}_{[X,Y]}\big)(V) = \\
&= -D^g_X\big(\G^{\n}(Y,V)\big) +\G^{\n}(Y,D^g_XV) +D^g_Y\big(\G^{\n}(X,V)\big) -\G^{\n}(X,D^g_YV) +\G^{\n}([X,Y],V) \\
&= -\big(D^g_X\G^{\n}\big)(Y,V) -\G^{\n}\big(D^g_XY,V\big) +\big(D^g_Y\G^{\n}\big)(X,V) +\G^{\n}(D^g_YX,V) +\G^{\n}\big([X,Y],V\big) \\
&= \big(D^g_Y\G^{\n}\big)(X,V) -\big(D^g_X\G^{\n}\big)(Y,V) \,\,,
\end{aligned}$$
the thesis follows.
\end{proof}

Let us choose now one of the Gauduchon connections $\nabla = \nabla^t$, $t \in \bR$. Then, its torsion $T^t(J,g)$ is related to the tensor field $D^g J$ by means of the following

\begin{lemma}
Fix an integer $k\geq1$ and a parameter $t \in \bR$. Then, there exists a constant $C=C(m,k,t)>1$ such that
\begin{equation} \label{eq:torsDJ}
\frac1C \sum_{j=1}^{k+1} \big| \big((D^g)^j J \big)_x \big|_g \leq \sum_{j=0}^{k} \big| \big((D^g)^j T^t(J,g) \big)_x \big|_g \leq C \sum_{j=1}^{k+1} \big| \big((D^g)^j J \big)_x \big|_g
\end{equation}
for any $x \in M$.
\end{lemma}

\begin{proof}
Easy computations show that
$$
\diff\w(X,Y,Z) = g((D^g_XJ)Y,Z) +g((D^g_YJ)Z,X) +g((D^g_ZJ)X,Y)
$$
and
$$
N_J(X,Y) = \big((D^g_{JX}J) -J\circ(D^g_XJ)\big)Y -\big((D^g_{JY}J) -J\circ(D^g_YJ)\big)X \,\, .
$$
Therefore, by \eqref{eq:tors}, it follows that there exists $C_1=C_1(m,k,t)$ such that
$$
\sum_{j=0}^{k+1} \big| \big((D^g)^j T^t(g) \big)_x \big|_g \leq C_1 \sum_{j=1}^{k+2} \big| \big((D^g)^j J \big)_x \big|_g
$$
for any $x \in M$. On the other hand, since $\n^tJ=0$, it holds
$$
D^g_XJ = D^g_XJ - \n^t_XJ = -[\G^t_X, J]
$$
and so, by \eqref{eq:tors-cont}, there exists $C_2=C_2(m,k,t)$ such that
$$
\sum_{j=1}^{k+2} \big| \big((D^g)^j J \big)_x \big|_g \leq C_2 \sum_{j=0}^{k+1} \big| \big((D^g)^j T^t(g) \big)_x \big|_g
$$ 
for any $x \in M$, which concludes the proof.
\end{proof}

Finally, formulas \eqref{eq:tors-cont}, \eqref{eq:diffcurvs} and \eqref{eq:torsDJ} directly imply the following result.

\begin{proposition} \label{prop:stime}
Fix an integer $k\geq0$, a parameter $t \in \bR$ and a point $x \in M$. \begin{itemize}
\item[i)] Let $K_1>0$ be such that
$$
\sum_{i=0}^{k} \big| \big((\n^t)^i \W^t(g) \big)_x \big|_g +\sum_{j=0}^{k+1} \big| \big((\n^t)^j T^t(g) \big)_x \big|_g < K_1 \,\, .
$$
Then, there exists a constant $C_1=C_1(m,k,t,K_1)>0$ such that
$$
\sum_{i=0}^{k} \big| \big((D^g)^i \Rm(g) \big)_x \big|_g + \sum_{j=1}^{k+2} \big| \big((D^g)^j J \big)_x \big|_g < C_1 \,\, .
$$
\item[ii)] Let $K_2>0$ be such that
$$
\sum_{i=0}^{k} \big| \big((D^g)^i \Rm(g) \big)_x \big|_g + \sum_{j=1}^{k+2} \big| \big((D^g)^j J \big)_x \big|_g < K_2 \,\, .
$$
Then, there exists a constant $C_2=C_2(m,k,t,K_2)>0$ such that
$$
\sum_{i=0}^{k} \big| \big((\n^t)^i \W^t(g) \big)_x \big|_g +\sum_{j=0}^{k+1} \big| \big((\n^t)^j T^t(g) \big)_x \big|_g < C_2 \,\, .
$$
\end{itemize}
\end{proposition}

\medskip
\section{Locally homogeneous almost-Hermitian spaces} \label{sec:locally-homog}
\setcounter{equation} 0

In this section, we will collect some known and less known facts about locally homogeneous almost-Hermitian spaces.
In particular, inspired by the Riemannian case, we briefly present the notions of real holomorphic Killing generators, of unitary transitive Lie algebras, and of Hermitian Ambrose-Singer connections.

\subsection{Real holomorphic Killing generators} \hfill \par

Let $(M,J,g)$ be an almost-Hermitian manifold. A vector field $X \in \cC^{\infty}(M,TM)$ is said to be {\it real holomorphic} (resp. {\it Killing}) if $\eL_XJ=0$ (resp. $\eL_Xg = 0$), namely, if its local flow is made by pseudo-holomorphic (resp. isometric) local transformations.
Moreover, for the sake of notation, we set
$$
A_X \= -D^gX
$$
for any vector field $X$ on $M$ (not necessarily real holomorphic or Killing).

We recall the following well-known fact, see e.g. {\cite[p 118]{nomizu}, \cite[p 541]{kostant}, see also \cite[Sect 2.1]{pediconi-geomded}, \cite[Lem I.1.4]{pediconi-thesis}} for detailed computations, stating that the space of pairs
$$ \big(X_p, (A_X)_p\big) \in T_pM \oplus \gs\go(T_pM,g_p) $$
at a point $p\in M$, varying $X$ real holomorphic Killing vector field of $(M,J,g)$ defined in a neighorhood of $p$, can be endowed with a structure of Lie algebra.

\begin{lemma}[{\cite[p 118]{nomizu}}] \label{lemma:lie-bracket}
Let $p \in M$ be a point and $X,Y$ real holomorphic Killing vector fields of the almost-Hermitian manifold $(M,J,g)$ defined on a neighborhood of $p$. Set the pairs $(v, A)\= \big(X_p,(A_X)_p\big)$, $(w,B)\=\big(Y_p,(A_Y)_p\big)$. Then, $[X,Y]$ is a real holomorphic Killing vector field of $(M,J,g)$ and
\begin{eqnarray}
\label{eq:Killbrack1} &&[X,Y]_p=A(w)-B(v) \,\, ,\\
\label{eq:Killbrack2} &&(A_{[X,Y]})_p=[A,B]+\Rm(g)_p(v,w) \,\, .
\end{eqnarray}
\end{lemma}

\begin{proof}
The Jacobi identity $\eL_{[X,Y]}=[\eL_X,\eL_Y]$ shows that $[X,Y]$ is real holomorphic Killing.
Equation \eqref{eq:Killbrack1} follows from the definition of torsion and $D^g$ being torsion-free.
Equation \eqref{eq:Killbrack2} follows by direct computation, by noticing that
$$
A_{[X,Y]} = [A_X,A_Y] +\Rm(g)(X,Y) +([D^g_X,A_Y]-\Rm(g)(X,Y)) -([D^g_Y,A_X]-\Rm(g)(Y,X))
$$
and that the quantity
$$
\alpha_Y(X,Z_1,Z_2) \= g([D^g_X,A_Y]Z_1-\Rm(g)(X,Y)Z_1, Z_2)
$$
vanishes, for $Y$ Killing vector field. Indeed, $\alpha_Y(X,Z_1,Z_2)$ is symmetric in $(X,Z_1)$ by using the algebraic Bianchi identity, and skew-symmetric in $(Z_1,Z_2)$ since $(\eL_Yg)(U,V)=g(D^g_VY,U)-g(D^g_UY,V)$ vanishes for $Y$ Killing vector field. This completes the proof.
\end{proof}

It is well known that Killing vector fields satisfy the following formulas:

\begin{lemma}[{\cite[Lem 10]{nomizu}}]\label{lemma:DkLg}
Let $X \in \cC^{\infty}(M,TM)$ be a Killing vector field on the Riemannian manifold $(M,g)$. Then
\begin{eqnarray}
&&A_X \cdot g = 0,\\
&&X \,\lrcorner\,\big( (D^g)^{k+1}\Rm(g) \big) + A_X \cdot \big( (D^g)^k \Rm(g) \big) = 0 \quad \text{ for any } k \in \bZ_{\geq0} \,\, ,
\label{eq:DkLg}
\end{eqnarray}
where the action of $A_X$ on the tensor bundle of $M$ is by derivation.
\end{lemma}

Futhermore, in the almost-Hermitian setting, we derive similar formulas for the infinitesimal action on the almost-complex structure in the following:

\begin{lemma}\label{lemma:DkLJ}
Let $X \in \cC^{\infty}(M,TM)$ be a Killing vector field on the almost-Hermitian manifold $(M,J,g)$. Then
\begin{equation} \label{eq:DkLJ}
(D^g)^k(\eL_XJ) =  X\,\lrcorner\, (D^g)^{k+1}J + A_X \cdot \big( (D^g)^kJ \big) \quad \text{ for any } k \in \bZ_{\geq0} \,\, ,
\end{equation}
where the action of $A_X$ on the tensor algebra is by derivation.
\end{lemma}

\begin{proof}
We prove the formula by induction on $k \in \bZ_{\geq0}$. For the sake of shortness of notation, we forget the metric $g$.

For $k=0$, take any vector field $Y \in \cC^{\infty}(M,TM)$ and compute (see \cite[Eqn 2.1.2]{kostant}):
$$\begin{aligned}
(\eL_XJ)(Y) &= \eL_X(JY)-J(\eL_XY)=[X,JY]-J[X,Y] \\
&= D_X(JY)-D_{JY}X-JD_XY+JD_YX \\
&= D_X(JY)-JD_XY+A_X(JY)-JA_XY \\
&= (D_XJ)(Y)+[A_X,J](Y) \,\, .
\end{aligned}$$

We also give explicit computations for $k=1$: for any vector field $Y \in \cC^{\infty}(M,TM)$,
$$\begin{aligned}
D(\eL_XJ)(Y, \_) &= D_{Y}D_XJ+D_{Y}[A_X,J] \\
&= D^2_{Y,X}J+D_{D_{Y}X}J+(D_{Y}A_X)J+A_X(D_{Y}J)-(D_{Y}J)A_X-J(D_{Y}A_X) \\
&= D^2_{X,Y}J-[\Rm(Y,X), J]+D_{D_{Y}X}J \\
& \qquad -\Rm(X,Y)J+A_X(D_{Y}J)-(D_{Y}J)A_X+J\Rm(X,Y) \\
&= D^2_{X,Y}J+D_{D_{Y}X}J+(A_X(DJ))(Y)+(DJ)(A_XY)-((DJ)A_X)(Y) \\
&= (X\,\lrcorner\,D^2J)(Y)+(A_X\cdot DJ)(Y),
\end{aligned}$$
where we used the Ricci formula \cite[Eqn (1.21)]{besse} $D^{2}_{X,Y}-D^2_{Y,X}=-\Rm(X,Y)$, and the property \cite[Lem 2.2]{kostant} $D_YA_X = -\Rm(X,Y)$ for Killing vector field $X$.

As for the inductive step, assume that Equation \eqref{eq:DkLJ} holds true for some $k\in\bZ_{\geq0}$. By using again the Kostant formula and the Ricci formulas, we compute, for vector fields $Y_0,Y_1,\ldots,Y_k\in\cC^{\infty}(M,TM)$,
$$\begin{aligned}
D^{k+1}_{Y_0, Y_1, \ldots, Y_k}(\eL_XJ) &= \big(D_{Y_0}(D^{k}(\eL_XJ))\big)_{Y_1,\ldots,Y_k} \\
&= \big(D_{Y_0}(X\,\lrcorner\, D^{k+1}J)\big)_{Y_1,\ldots,Y_k}+\big(D_{Y_0}(A_X \cdot D^kJ)\big)_{Y_1,\ldots,Y_k} \\
&= D^{k+2}_{Y_0,X,Y_1,\ldots,Y_k}J+D^{k+1}_{D_{Y_0}X,Y_1,\ldots,Y_k}J\\
&\qquad +\big((D_{Y_0}A_X)(D^kJ)\big)_{Y_1,\ldots,Y_k}+\big(A_X(D_{Y_0}D^{k}J)\big)_{Y_1,\ldots,Y_k} \\
&\qquad -\big(D_{Y_0}D^kJ)A_X\big)_{Y_1,\ldots,Y_k}-\big((D^kJ)(D_{Y_0}A_X)\big)_{Y_1,\ldots,Y_k} \\
&= D^{k+2}_{X,Y_0,Y_1,\ldots,Y_k}J-\Rm(Y_0,X) \cdot D^{k}_{Y_1,\ldots,Y_k}J \\
&\qquad +D^{k}_{\Rm(Y_0,X)Y_1,\ldots,Y_k}J+\cdots+D^{k}_{Y_1,\ldots,\Rm(Y_0,X)Y_k}J +D^{k+1}_{D_{Y_0}X,Y_1,\ldots,Y_k}J \\
&\qquad - ( \Rm(X,Y_0) \cdot D^kJ )_{Y_1,\ldots,Y_k}+( A_X \cdot D^{k+1}J )_{Y_0,Y_1,\ldots,Y_k}+D^{k+1}_{A_XY_0,Y_1,\ldots,Y_k}J \\
&= (X\,\lrcorner\,D^{k+2}J)_{Y_0,Y_1,\ldots,Y_k}+ ( A_X \cdot D^{k+1}J )_{Y_0,Y_1,\ldots,Y_k} \,\, ,
\end{aligned}$$
completing the proof.
\end{proof}

We recall now the following

\begin{definition}
An almost-Hermitian manifold $(M,J,g)$ is said to be a {\it locally homogeneous almost-Hermitian space} if its pseudogroup of local automorphisms $\eP^{J,g}$ acts transitively, that is, for any $x, y \in M$ there exist neighborhoods $U_x,U_y \subset M$ of $x$, $y$, respectively, and a local pseudo-holomorphic isometry $f : U_x \rightarrow U_y$ such that $f(x) = y$.
\end{definition}

Let $(M,J,g)$ be a locally homogeneous almost-Hermitian space. Since $J,g$ determine a smooth $\fU(m)$-structure on $M$, it follows that $\eP^{J,g}$ is a transitive {\it Lie pseudogroup of transformations on $M$} and so, by standard Lie pseudogroup theory (see e.g. \cite[Thm 2.2]{spiro}), the following result holds.

\begin{lemma}\label{lemma:analyticity}
Let $(M,J,g)$ be a locally homogeneous almost-Hermitian space. Then both $g$ and $J$ are real-analytic.
\end{lemma}

Following \cite[page 110]{nomizu}, we give the following definition. Given a locally homogeneous almost-Hermitian space $(M,J,g)$ and a distinguished point $p \in M$, the {\it real holomorphic Killing generators at $p$} are defined as those pairs $(v,A) \in T_pM \oplus \gg\gl(T_pM)$ such that
\begin{equation} \label{killgen}
\begin{gathered}
A \cdot g_p = 0 \,\, , \quad v \,\lrcorner\, \big((D^g)^{k+1}J\big){}_p + A \cdot \big((D^g)^kJ\big){}_p = 0 \,\, , \\
v \,\lrcorner\, \big((D^g)^{k+1}\Rm(g)\big){}_p + A \cdot \big((D^g)^k\Rm(g)\big){}_p = 0 \quad \text{ for any $k \in \bZ_{\geq0}$} \,\, ,
\end {gathered}
\end{equation}
where $\gg\gl(T_pM)$ acts on the tensor algebra over $T_pM$ as a derivation.
This definition is suggested by Lemma \ref{lemma:DkLg} and Lemma \ref{lemma:DkLJ}, as shown in the following

\begin{proposition}\label{prop:killing-nomizu}
If $X$ is a real holomorphic Killing vector field of $(M,J,g)$ defined in a neighborhood of the point $p\in M$, then the pair $(v,A) \= \big(X_p,(A_X)_p\big)$ is a real holomorphic Killing generator of $(M,J,g)$ at $p$. Conversely, there exists a neighborhood $\W_p \subset M$ of $p$ such that, for any holomorphic Killing generator $(v,A)$ at $p$, there exists a real holomorphic Killing vector field $X$ of $(M,J,g)$ defined on $\W_p$ such that $(v,A) = \big(X_p,(A_X)_p\big)$.
\end{proposition}

\begin{proof}
Assume that $X$ is a real holomorphic Killing vector field of $(M,J,g)$ defined in a neighborhood of $p$. Then, by the very definition, by \eqref{eq:DkLg}, and by \eqref{eq:DkLJ} respectively, it follows that $\big(X_p,(A_X)_p\big)$ satisfies \eqref{killgen}, namely, $\big(X_p,(A_X)_p\big)$ is a real holomorphic Killing generator of $(M,J,g)$ at $p$.

Conversely, being $g$ real-analytic, then, by \cite[Thms 1, 2]{nomizu}, there exists a neighborhood $\W_p$ of $p$ such that, for any real holomorphic Killing generator $(v,A)$ at $p$, one can find a real-analytic Killing vector field $X$ defined on $\W_p$ such that $(v,A) = \big(X_p,(A_X)_p\big)$. Moreover, by means of \eqref{eq:DkLJ}, it follows that $(D^g)^k(\eL_XJ){}_p = 0$ for any $k \in \bZ_{\geq0}$. Since the endomorphism field $\eL_XJ$ is real-analytic, we conclude that $X$ is real holomorphic. 
\end{proof}

Thanks to Lemma \ref{lemma:lie-bracket}, we denote by $\gkill^{J,g}$ the Lie algebra of all the real holomorphic Killing generators of the locally homogeneous almost-Hermitian manifold $(M,J,g)$ at the point $p$ with the Lie bracket 
\begin{equation}
\big[(v,A),(w,B)\big] \= \big(A(w)-B(v),[A,B]+\Rm(g)_p(v,w)\big)
\end{equation}
and we call it the {\it Hermitian Nomizu algebra of $(M,J,g)$ at $p$}.

\subsection{Unitary transitive Lie algebras} \hfill \par

We recall that the {\it Malcev-closure} in the connected Lie group $\fG$ of a Lie subalgebra $\gh$ of $\gg=\Lie(\fG)$ is the Lie algebra of the closure $\overline{\fH}$ of $\fH$ in $\fG$, where $\fH$ is the simply connected Lie group with $\Lie(\fH)=\gh$.
Following \cite{lauret-jlms, pediconi-geomded}, we consider the following
\begin{definition} \label{utLa}
Let $m,q \in \bZ_{\geq0}$. A {\it unitary transitive Lie algebra $(\gg=\gh+\gm,I,\langle\,,\rangle)$ of rank $(m,q)$} is the datum of \begin{itemize}
\item[$\bcdot$] a $(q{+}2m)$-dimensional Lie algebra $\gg$;
\item[$\bcdot$] a $q$-dimensional Lie subalgebra $\gh \subset \gg$ which does not contain any non-trivial ideal of $\gg$;
\item[$\bcdot$] an $\ad(\gh)$-invariant complement $\gm$ of $\gh$ in $\gg$;
\item[$\bcdot$] an $\ad(\gh)$-invariant linear complex structure $I$ on $\gm$;
\item[$\bcdot$] an $\ad(\gh)$-invariant Euclidean product $\langle\,,\rangle$ on $\gm$ such that $\langle I{\cdot},I{\cdot\cdot}\rangle=\langle{\cdot},{\cdot\cdot}\rangle$.
\end{itemize}
\end{definition}

\noindent A unitary transitive Lie algebra $(\gg=\gh+\gm,I,\langle\,,\rangle)$ is said to be
\begin{itemize}
\item[$\bcdot$] {\it integrable} it the linear complex structure $I$ satisfies
$$
[IX,IY]_{\gm}-[X,Y]_{\gm} = I[IX,Y]_{\gm}+I[X,IY]_{\gm} \quad \text{ for any $X,Y \in \gm$\,\,,}
$$
{\it non-integrable} otherwise;
\item[$\bcdot$] {\it regular} if $\gh$ is Malcev-closed in the simply connected Lie group $\fG$ with $\Lie(\fG)=\gg$, {\it non-regular} otherwise.
\end{itemize}

Let $(\gg=\gh+\gm,I,\langle\,,\rangle)$ be a unitary transitive Lie algebra of rank $(m,q)$. Since there are no ideals of $\gg$ in $\gh$, the adjoint action of $\gh$ on $\gm$ is a faithful representation in $\gu(\gm,I,\langle\,,\rangle)$ and so $0 \leq q \leq m^2$. An {\it adapted frame} is a basis $u=(e_1,{\dots},e_{q+2m}): \bR^{q+2m} \to \gg$ such that
$$
\gh=\vspan(e_1,{\dots},e_q) \,\, , \quad \gm=\vspan(e_{q+1},{\dots},e_{q+2m}) \,\, , \quad Ie_{q+(2i-1)} = e_{q+2i} \,\, , \quad \langle e_{q+i},e_{q+j}\rangle=\d_{ij} \,\, .
$$

An {\it isomorphism} between two unitary transitive Lie algebras $(\gg_i=\gh_i+\gm_i,I_i,\langle\,,\rangle_i)$ is any Lie algebra isomorphism $\f:\gg_1 \to \gg_2$ such that
$$
\f(\gh_1)= \gh_2 \,\, , \quad \f(\gm_1)=\gm_2 \,\, , \quad 
I_2 \circ \f|_{\gm_1} = \f|_{\gm_1} \circ I_1 \,\, , \quad \langle\,,\rangle_1=(\f|_{\gm_1})^*\langle\,,\rangle_2 \,\, .
$$

\begin{remark} \label{rem:Sp}
The product $\langle \, , \, \rangle$ on $\gm$ can be extended to an inner product $\langle \, , \, \rangle'$ on $\gg$ in such a way that the decomposition $\gg=\gh+\gm$ is orthogonal and $\langle \, , \, \rangle'\vert_{\gh\otimes\gh}$ corresponds, via the embedding $\gh \stackrel{\ad}{\to} \gu(\gm, I, \langle \, , \, \rangle) \hookrightarrow \so(\gm, \langle \, , \, \rangle)$, to the negative Cartan-Killing form of $\so(\gm, \langle \, , \, \rangle)$.
\end{remark}

A distinguished class of unitary transitive Lie algebras are given by the Hermitian Nomizu algebras of locally homogeneous almost-Hermitian manifolds. Indeed, let $(M,J,g)$ be a locally homogeneous Hermitian space, $p \in M$ a distinguished point and $\gkill^{J,g}$ the Hermitian Nomizu algebra of $(M,J,g)$ at $p$. Consider the Euclidean scalar product on $\gkill^{J,g}$ given by $$\langle\!\langle(v,A),(w,B)\rangle\!\rangle_g\=g_p(v,w)-\tr(AB) \,\, ,$$ set $\gkill^{J,g}_0\=\{(0,A) \in \gkill^{J,g}\} \subset \gu(T_pM,J_p,g_p)$ and let $\gm^g$ be the $\langle\!\langle\,,\rangle\!\rangle_g$-orthogonal complement of $\gkill^{J,g}_0$ in $\gkill^{J,g}$. Being $(M,J,g)$ locally homogeneous, it follows that $\gm^g \simeq T_pM$ and this allows us to define a linear complex structure $I_J$ and a scalar product $\langle\,,\rangle_g$ on $\gm^g$ induced by $J$ and $g$ on $M$, respectively. Then,
\begin{equation}\label{eq:dec-nomizu}
(\gkill^{J,g}=\gkill^{J,g}_0+\gm^g,I_J,\langle\,,\rangle_g)
\end{equation}
is a unitary transitive Lie algebra. It is straightforward to check that the Hermitian Nomizu algebra, modulo isomorphisms, does not depend on the particular choice of the point $p$ and that two locally homogeneous almost-Hermitian spaces are locally pseudo-holomorphically isometric if and only if their Hermitian Nomizu algebras are isomorphic.

\subsection{Hermitian Ambrose--Singer connections and Hermitian--Singer invariant} \label{sec:h-singer} \hfill \par

As in the Riemannian case, the following facts hold true.
First, locally homogeneity is encoded by the existence of a distinguished connection. More precisely

\begin{theorem}[{\cite{kiricenko, sekigawa}}]
Let $(M,J,g)$ be an almost-Hermitian manifold.
It is locally homogenous if and only if it admits a Hermitian connection with parallel torsion and parallel curvature.
\end{theorem}

A connection as in the previous theorem is called a {\em Hermitian Ambrose-Singer connection}.
Moreover, by the proof of this theorem, it follows that Hermitian Ambrose-Singer connections are in one-to-one correspondence with the choice of a reductive decomposition for the Hermitian Nomizu algebra, i.e. a choice of a complement $\gm$ for the isotropy algebra $\gkill_0^{J,g}$ inside $\gkill^{J,g}$ that is invariant under the adjoint representation. One of this choice has already been discussed in the previous section.

Second, it is possible to recognize a locally homogeneous almost-Hermitian manifold by means of a finite set of algebraic tensors on a tangent space.
To this purpose, fix $p \in M$ and for any $k\geq0$ set
\begin{equation}
\gj(k) \= \big\{A \in \gu(T_pM,J_p,g_p) : \, A \cdot \big((D^g)^i\Rm(g)_p\big)=0 \,\, \text{ for } \,\, 0 \leq i \leq k \,\,, \,\, A \cdot \big((D^g)^jJ_p\big)=0 \,\, \text{ for } 1 \leq j \leq k+2 \big\} \,\, .
\end{equation}
Since $\big(\gj(k)\big)_{k \in \bZ_{\geq0}}$ is a filtration of the finite dimensional Lie algebra $\gu(T_pM,J_p,g_p)$, there exists a first integer $k_{J,g}$ such that $\gj(k_{J,g})=\gj(k_{J,g}{+}1)$. It is called the {\it Hermitian-Singer invariant of $(M,J,g)$} by \cite{console-nicolodi}.
Notice that, by adapting \cite[Thm 4.1]{tricerri}, whose proof can be found in \cite[proof of Thm 2.1]{nicolodi-tricerri}, and \cite[Prop 4.3]{tricerri}, it is possible to prove that $\gj(k)=\gj(k_{J,g})$ for any $k\geq k_{J,g}$.

\begin{remark}[Open question, 1] \label{rmk:open-1}
In the same spirit of \cite{meusers}, it will be interesting to construct examples of: \begin{itemize}
\item[i)] locally homogeneous almost-Hermitian spaces with arbitrarily high Hermitian Singer invariant;
\item[ii)] pairs of locally homogeneous almost-Hermitian spaces with Hermitian Singer invariant $k$, that are not locally pseudo-holomorphically isometric, which have the same Riemannian curvature up to order $k$ and the same almost-complex structure up to order $k+2$.
\end{itemize}
\end{remark}

For later purposes, for any positive integer $m$ we set
\begin{equation}
\jmath(m) \= \max\{k_{J,g} : \text{$(M,J,g)$ alm. Herm. loc. hom. with \,$\dim_{\bR} M\leq 2m$} \} \label{i(m)} \,\, .
\end{equation}
Notice that $m \mapsto \jmath(m)$ is non-decreasing and $0 \leq \jmath(m) \leq m^2-1$.

For any $m,s \in \bN$ with $s \geq \jmath(m)+2$, we define $\wt{\cX}^s(m)$ to be the set of all the $(2s{+}3)$-tuples
$$\begin{gathered}
(J^1,{\dots},J^{s+2}) \oplus (R^0,R^1,{\dots},R^s) \in E^{\mathsmaller{(1)}}(m,s) \oplus E^{\mathsmaller{(2)}}(m,s) \,\, , \quad \text{ with } \\
E^{\mathsmaller{(1)}}(m,s) \= \bigoplus_{1 \leq k \leq s+2}\Big({\textstyle\bigotimes^k}(\bR^{2m})^* \otimes \so(2m)\Big) \,\, , \quad 
E^{\mathsmaller{(2)}}(m,s) \= \bigoplus_{0 \leq k \leq s}\Big({\textstyle\bigotimes^k}(\bR^{2m})^*\otimes\L^2(\bR^{2m})^*\otimes\so(2m)\Big) \,\, .
\end{gathered}$$
satisfying the subsequent conditions (X1) and (X2). \\[8pt]
(X1) The following eight identities hold:
\begin{align*}
\text{i)}&\,\, \langle R^0(Y_1{\wedge}Y_2)V_1,V_2 \rangle_{\st} = \langle R^0(V_1{\wedge}V_2)Y_1,Y_2 \rangle_{\st} \,\, , 	\\
\text{ii)}&\,\, \gS_{{}_{Y_1,Y_2,V_1}} \langle R^0(Y_1{\wedge}Y_2)V_1,V_2 \rangle_{\st}=0 \,\, , \\
\text{iii)}&\,\, \langle R^1(X_1|Y_1{\wedge}Y_2)V_1,V_2 \rangle_{\st} = \langle R^1(X_1|V_1{\wedge}V_2)Y_1,Y_2 \rangle_{\st} \,\, , \\
\text{iv)}&\,\, \gS_{{}_{Y_1,Y_2,V_1}} \langle R^1(X_1|Y_1{\wedge}Y_2)V_1,V_2\rangle_{\st}=0 \,\, , \\
\text{v)}&\,\, \gS_{{}_{X_1,Y_1,Y_2}} \langle R^1(X_1|Y_1{\wedge}Y_2)V_1,V_2\rangle_{\st}=0 \,\, , \\
\text{vi)}&\,\, R^{k+2}(X_1,X_2,X_3,{\dots}X_{k+2}|Y_1{\wedge}Y_2)-R^{k+2}(X_2,X_1,X_3,{\dots}X_{k+2}|Y_1{\wedge}Y_2) \\
&\hskip 130pt =-\big(R^0(X_1{\wedge}X_2) \cdot R^k\big)(X_3,{\dots}X_{k+2}|Y_1{\wedge}Y_2) \quad \text{ for any $0 \leq k \leq s-2$ } \,\, , \\
\text{vii)}&\,\, J^2(X_1,X_2) - J^2(X_2,X_1) = -R^0(X_1{\wedge}X_2) \cdot I_{\st} \,\, , \\
\text{viii)}&\,\, J^{k+2}(X_1,X_2,X_3,{\dots}X_{k+2}) - J^{k+2}(X_2,X_1,X_3,{\dots}X_{k+2}) \\
&\hskip 130pt =-\big(R^0(X_1{\wedge}X_2) \cdot J^k\big)(X_3,{\dots}X_{k+2}|Y_1{\wedge}Y_2) \quad \text{ for any $1 \leq k \leq s$ } \,\, ,
\end{align*}
where $\so(2m)$ acts on the tensor algebra on $\bR^{2m}$ by derivation. \\[8pt]
(X2) For any $1\leq k\leq s$, the maps
$$\begin{array}{ll}
\a^k(A) \= (A \cdot I_{\st}, A \cdot J^1, {\dots}, A \cdot J^{k+1}) \oplus (A \cdot R^0, A \cdot R^1, {\dots}, A \cdot R^{k-1}) \,\, , & \,\, \text{ with } A \in \so(2m) \,\,, \\[3pt]
\b^k(X) \= (X \lrcorner J^1, {\dots}, X \lrcorner J^{k+2}) \oplus (X \lrcorner R^1, X \lrcorner R^2, {\dots} , X \lrcorner R^k) \,\, , & \,\, \text{ with } X \in \bR^{2m} 
\end{array}$$
verify
$$\begin{array}{ll}
\b^k(\bR^{2m}) \subset \a^{k-1}(\gs\go(2m)) & \,\, \text{ for any } \jmath(m)+2\leq k \leq s \,\, , \\[3pt]
\ker(\a^k) = \ker(\a^{k+1}) & \,\, \text{ for any } \jmath(m) \leq k \leq s-1 \,\, .
\end{array}$$ \vskip 8pt

Notice that $\wt{\cX}^s(m)$ is invariant under the standard left action of $\fU(m)$, and hence

\begin{definition} Let $m,s \in \bN$ with $s \geq \jmath(m)+2$. We call {\it Hermitian $s$-tuples of rank $m$} the elements of the quotient space $\cX^s(m) \= \fU(m) \backslash \widetilde{\cX}^s(m)$. \end{definition}

This definition is motivated by the following result, which is the almost-Hermitian analogue of \cite[Thm 3.1]{nicolodi-tricerri}:

\begin{theorem}[{\cite{console-nicolodi}}] \label{thm:CNcurvmod}
Let $(M^{2m},J,g)$ be a locally homogeneous almost-Hermitian space. Let also $p \in M$ be a point, $u: \bR^{2m} \to T_pM$ a unitary frame and $s \geq \jmath(m)+2$ an integer. Then
$$
\big(u^*\big(D^gJ\big){}_p, {\dots}, u^*\big((D^g)^{s+2}J\big){}_p\big) \oplus \big(u^*\Rm(g)_p, u^*\big(D^g\Rm(g)\big){}_p, {\dots}, u^*\big((D^g)^s\Rm(g)\big){}_p\big)
$$
defines a Hermitian $s$-tuple of rank $m$ which is independent of $p$ and $u$. Conversely, for any Hermitian $s$-tuple $\theta^s \in \cX^s(m)$ of rank $m$, there exists a locally homogeneous almost-Hermitian space $(M^m,J,g)$, uniquely determined up to a local pseudo-holomorphic isometry, such that
$$
\theta^s = \big[\big(u^*\big(D^gJ\big){}_p, {\dots}, u^*\big((D^g)^{s+2}J\big){}_p\big) \oplus \big(u^*\Rm(g)_p, u^*\big(D^g\Rm(g)\big){}_p, {\dots}, u^*\big((D^g)^s\Rm(g)\big){}_p\big)\big]
$$
for some $p \in M$ and $u: \bR^{2m} \to T_pM$ unitary frame.
\end{theorem}

\section{Almost-Hermitian geometric models} \label{sec:geom-models} 
\setcounter{equation} 0

\subsection{The class of almost-Hermitian geometric models} \hfill \par

In this section, following \cite{bohm-lafuente-simon, pediconi-annsns}, we introduce a special class of locally homogeneous almost-Hermitian spaces, namely

\begin{definition} \label{def:GM}
A {\it $2m$-dimensional almost-Hermitian geometric model} is a locally homogeneous almost-Hermitian distance ball $(\eB, \hat{J}, \hat{g})=(\eB_{\hat{g}}(o,\pi), \hat{J}, \hat{g})$ of radius $\pi$, dimension $\dim_{\bR}\eB=2m$, with bounded sectional curvature $|\sec(\hat{g})| \leq 1$ and injectivity radius at the center $o \in \eB$ equal to $\inj_o(\eB, \hat{g})= \pi$.
\end{definition}

From now on, up to pulling back the metric via the Riemannian exponential map $\Exp(\hat{g})_o$, any almost-Hermitian geometric model will be always assumed to be of the form $(B^{2m},\hat{J},\hat{g})$, where $B^{2m} \= B_{\st}(0,\pi) \subset \bR^{2m}$ is the $2m$-dimensional Euclidean ball of radius $\pi$, the standard coordinates of $B^{2m}$ will be always assumed to be normal for $\hat{g}$ at $0$ and $\hat{J}|_0 = I_{\st}$.
In particular, the geodesics starting from $0 \in B^{2m}$ are precisely the straight lines and the Riemannian distance from the center equals $\td_{\hat{g}}(0,x)=|x|_{\st}$ for any $x \in B^{2m}$. Hence, $\eB_{\hat{g}}(0,r)=B_{\st}(0,r)$ for any $0<r \leq \pi$. \smallskip

For latter purposes, we prove that local pseudo-holomorphic isometries can be extended in the following way:

\begin{lemma} \label{lemma_unifext1}
Let $(B^{2m}, \hat{J}_1, \hat{g}_1)$ and $(B^{2m}, \hat{J}_2, \hat{g}_2)$ be two almost-Hermitian geometric models and ssume that there exists $0<\e<\pi$ and a pointed pseudo-holomorphic isometry $f: (B_{\st}(0,\e), \hat{J}_1, \hat{g}_1) \to (B_{\st}(0,\e), \hat{J}_2, \hat{g}_2)$. Then, $f$ extends analytically to a pointed pseudo-holomorphic isometry $\tilde{f}: (B^{2m}, \hat{J}_1, \hat{g}_1) \to (B^{2m}, \hat{J}_2, \hat{g}_2)$.
\end{lemma}

\begin{proof}
Let us define the map
$$
\tilde{f}: B^{2m} \to B^{2m} \,\, , \quad \tilde{f} \= \Exp(\hat{g}_2)_{0} \circ \diff f |_{0} \circ \Exp(\hat{g}_1)_{0}^{-1} \,\, .
$$
Then, by construction, it follows that $\tilde{f}$ is real-analytic diffeomorphism satisfying $\tilde{f}(x)=f(x)$ for any $x \in B_{\st}(0,\e)$ and $\tilde{f}^*\hat{g}_2 = \hat{g}_1$. We now consider the function
$$
h : B^{2m} \to \bR \,\, , \quad h(x) \= \big|\diff \tilde{f}(x) \circ J_1(x) \circ (\diff \tilde{f}(x))^{-1} - J_2(\tilde{f}(x))\big|_{\hat{g}_2}^{2}
$$
and we observe that it is real-analytic. Moreover $h(x)=0$ for any $x \in B_{\st}(0,\e)$ and so it follows that $h(x)=0$ for any $x \in B^{2m}$. This completes the proof.
\end{proof}

Following \cite[Lemma 1.3 and Lemma 1.4]{bohm-lafuente-simon} and the proof of Lemma \ref{lemma_unifext1}, one can also prove the following

\begin{lemma} \label{lemma_unifext2}
Let $(B^{2m}, \hat{J}, \hat{g})$ be an almost-Hermitian geometric model. Then
\begin{equation} \label{eq_inj}
\inj_x(B^{2m}, \hat{g})= \pi -|x|_{\st} \quad \text{ for any } x \in B^{2m} \,\, .
\end{equation}
Moreover, fix $x, y \in B^{2m}$ and set $r_{x,y} \= \pi - \max\{|x|_{\st},|y|_{\st}\}$. Then, any pointed pseudo-holomorphic iso\-me\-try $f: (\eB_{\hat{g}}(x,\e),\hat{J}, \hat{g}) \to (\eB_{\hat{g}}(y,\e),\hat{J}, \hat{g})$ can be uniquely extended to a pointed pseudo-holomorphic isometry $\tilde{f}: (\eB_{\hat{g}}(x,r_{x,y}),\hat{J}, \hat{g}) \to (\eB_{\hat{g}}(y,r_{x,y}),\hat{J}, \hat{g})$.
\end{lemma}

One of the main properties of the class of almost-Hermitian geometric models is the fact that they give rise to a good parametrization for the moduli space of locally homogeneous almost-Hermitian spaces up to local pseudo-holomorphic isometries. More precisely, the following existence result holds true.

\begin{theorem} \label{thm:existence}
Let $(M^{2m},J,g)$ be a locally homogeneous almost-Hermitian space with $|\sec(g)|\leq1$. Then, there exists a $2m$-dimensional almost-Hermitian geometric model $(B^{2m}, \hat{J}, \hat{g})$ that is locally pseudo-holomorphically isometric to $(M^{2m},J,g)$. The almost-Hermitian geometric model is unique up to pseudo-holomorphic isometry.
\end{theorem}

\begin{proof}
Fix a point $p\in M$. By \cite[Theorem A]{pediconi-annsns}, there exists a $2m$-dimensional, smooth, locally homogeneous Riemannian distance ball $(B^{2m}, \hat{g})$ with $|\sec(\hat{g})| \leq 1$ and $\inj_0(B^{2m}, \hat{g})= \pi$, together with a smooth diffeomorphism
$$
\phi: B_{\st}(0,\e) \subset B^{2m} \to \eU \subset M
$$
verifying $\phi(0)=p$ and $\phi^*g = \hat{g}$. Then, it is easy to check that $\hat{J} \= (\diff \phi)^{-1} \circ J \circ \diff \phi$ can be extended to the whole ball $B^{2m}$ and gives rise to an almost-Hermitian geometric model $(B^{2m}, \hat{J}, \hat{g})$. Finally, the uniqueness follows from Theorem \ref{thm:CNcurvmod} and Lemma \ref{lemma_unifext1}.
\end{proof}

\subsection{Cheeger-Gromov convergence of almost-Hermitian geometric models} \hfill \par

In the Riemannian setting, geometric models are introduced to provide a right framework to study convergence in the Cheeger-Gromov topology even without a lower bound on the injectivity radius.
Indeed, it is well-known by \cite{cheeger-gromov-1, cheeger-gromov-2} that there exist families of Riemannian manifolds that collapse with bounded curvature.
The idea of studying limits of such families in the Cheeger-Gromov topology was originally conceived in the seminal works by \cite{glickenstein, lott}, where the notion of {\em Riemannian groupoids} is used.
Remarkably, when restricting to Riemannian homogeneous spaces, this construction reduces to consider geometric models.

Firstly, we give the following definition of convergence, which generalizes the usual notion of pointed convergence for complete Riemannian manifolds (see e.g. \cite{petersen}) to the case of incomplete almost-Hermitian manifolds. In the following, the Banach spaces $\cC^{k,\a}(\ol{B})$ are defined following \cite[p. 52]{gilbarg-trudinger} for any bounded ball $B \subset \bR^{2m}$.

\begin{definition} \label{defpointed}
A sequence $(B^{2m},\hat{J}^{(n)},\hat{g}^{(n)})$ of $2m$-dimensional almost-Hermitian geometric models is said to {\it converge in the pointed $\cC^{k,\a}$-topology} to a $2m$-dimensional almost-Hermitian geometric model $(B^{2m},\hat{J}^{(\infty)},\hat{g}^{(\infty)})$ if, for any $0<\d < \pi$, there exists a sequence of $\cC^{k+1,\a}$-embeddings $\phi_{\d}^{(n)}: B_{\st}(0,\pi-\d) \to B^{2m}$ such that $\phi_{\d}^{(n)}(0)=0$ for any $n \in \bN$ and
$$
\Big\|\big(\phi_{\d}^{(n)*}\hat{g}^{(n)}\big)_{ij} -\big(\hat{g}^{(\infty)}\big)_{ij}\Big\|_{\cC^{k,\a}(\overline{B_{\st}(0,\pi-\d)})} \to 0 \,\, , \quad
\Big\|\big((\diff \phi_{\d}^{(n)})^{-1} \circ \hat{J}^{(n)} \circ (\diff \phi_{\d}^{(n)})\big)^i_j -\big(\hat{J}^{(\infty)}\big)^i_j\Big\|_{\cC^{k,\a}(\overline{B_{\st}(0,\pi-\d)})} \to 0
$$
as $n \to +\infty$, for any $1 \leq i, j \leq 2m$.
\end{definition}

Then we observe that, by means of \cite[Corollary 3.8]{pediconi-annsns} and Proposition \ref{prop:stime}, the following convergence result holds true in the set of all the almost-Hermitian geometric models.

\begin{theorem} \label{thm:conv}
Let $(B^{2m},\hat{J}^{(n)},\hat{g}^{(n)})$ be a sequence of $2m$-dimensional almost-Hermitian geometric models and assume that there exist an integer $k \geq 0$, a parameter $t \in \bR$ and a constant $K>0$ such that, for any $n\in\bN$,
$$
\sum_{i=0}^{k} \big| (\n^{(n)\,t})^i \W^t(\hat{J}^{(n)},\hat{g}^{(n)}) \big|_{\hat{g}^{(n)}} +\sum_{j=0}^{k+1} \big| (\n^{(n)\,t})^j T^t(\hat{J}^{(n)},\hat{g}^{(n)}) \big|_{\hat{g}^{(n)}} < K \,\, .
$$
Then, $(B^{2m},\hat{J}^{(n)},\hat{g}^{(n)})$ subconverges to a limit $2m$-dimensional almost-Hermitian geometric model $(B^{2m},\hat{J}^{(\infty)},\hat{g}^{(\infty)})$ in the pointed $\cC^{k+1,\a}$-topology, for any $0<\a<1$.
\end{theorem}

\begin{proof}
By Proposition \ref{prop:stime}, there exists a constant $C=C(m,k,t,K)$ such that, for any $n\in\bN$,
\begin{equation}\label{eq:stima-proof}
\sum_{i=0}^{k} \big| \big((D^{\hat{g}^{(n)}})^i \Rm(\hat{g}^{(n)}) \big) \big|_{\hat{g}^{(n)}} < C \,\,, \qquad \sum_{j=1}^{k+2} \big| \big((D^{\hat{g}^{(n)}})^j \hat{J}^{(n)} \big) \big|_{\hat{g}^{(n)}} < C \,\, .
\end{equation}
By the first inequality in \eqref{eq:stima-proof} and \cite[Corollary 3.8]{pediconi-annsns}, up to pass to a subsequence, we can assume that the sequence of Riemannian distance balls $(B^{2m},\hat{g}^{(n)})$ converges to a $2m$-dimensional, smooth, locally homogeneous Riemannian distance ball $(B^{2m}, \hat{g}^{(\infty)})$ in the pointed $\cC^{k+1,\a}$-topology, for any $0<\a<1$, with $|\sec(\hat{g}^{(\infty)})| \leq 1$ and $\inj_0(B^{2m}, \hat{g}^{(\infty)})= \pi$. In other words, we can fix any $0<\d<\pi$ and find a sequence of $\cC^{k+2,\a}$-embeddings $\phi_{\d}^{(n)}: B_{\st}(0,\pi-\d) \to B^{2m}$ such that $\phi_{\d}^{(n)}(0)=0$ and
\begin{equation}\label{eq:chefatica}
\Big\|\big(\phi_{\d}^{(n)*}\hat{g}^{(n)}\big)_{ij} -\big(\hat{g}^{(\infty)}\big)_{ij}\Big\|_{\cC^{k+1,\a}(\overline{B_{\st}(0,\pi-\d)})} \to 0 \quad \text{ as } n \to +\infty \,\, , \,\, \text{ for any } 1\leq i, j \leq 2m \,\,.
\end{equation}
Actually, the intertwining embeddings $\phi_{\d}^{(n)}$ can be assumed to be smooth.
Indeed, we can approximate each $\phi_{\d}^{(n)}$ with a smooth map $\tilde\phi_{\d}^{(n)}: B_{\st}(0,\pi-\d) \to B^{2m}$, satisfying $\tilde\phi_{\d}^{(n)}(0)=0$, in the $\cC^{k+2,\a}$-norm, i.e.
$$
\max_{1 \leq \ell \leq 2m} \Big\| \big(\tilde{\phi}_{\d}^{(n)} - \phi_{\d}^{(n)}\big)^{\ell} \Big\|_{\cC^{k+2,\a}(\overline{B_{\st}(0,\pi-\d)})} \leq \e^{(n)}
$$
for some constant $\e^{(n)}>0$. Notice that the condition of being embedding is open, see e.g. \cite[Ch 2, Thm 1.4]{hirsch}, and so, up to take $\e^{(n)}$ small enough, the map $\tilde\phi_{\d}^{(n)}$ is an embedding as well. Moreover, a direct computation shows that there exists a constant $C>0$, that does not depend on $n$, such that
$$
\Big\|\big(\tilde{\phi}_{\d}^{(n)*}\hat{g}^{(n)}\big)_{ij} -\big(\hat{g}^{(\infty)}\big)_{ij}\Big\|_{\cC^{k+1,\a}(\overline{B_{\st}(0,\pi-\d)})} \leq C
\Big(\e^{(n)} + \Big\|\big(\phi_{\d}^{(n)*}\hat{g}^{(n)}\big)_{ij} -\big(\hat{g}^{(\infty)}\big)_{ij}\Big\|_{\cC^{k+1,\a}(\overline{B_{\st}(0,\pi-\d)})}\Big)
$$
for any $1 \leq i, j \leq 2m$. Therefore, letting $\e^{(n)} \to 0$, this shows that \eqref{eq:chefatica} holds also with $\tilde\phi_{\d}^{(n)}$ in place of $\phi_{\d}^{(n)}$.

By \eqref{eq:stima-proof}, the tensors $(\diff \phi_{\d}^{(n)})^{-1} \circ \hat{J}^{(n)} \circ (\diff \phi_{\d}^{(n)})$ are uniformly bounded in the $\cC^{k+2}$-norm on the compact set $\overline{B_{\st}(0,\pi-\d)}$. Then, by the Ascoli-Arzel\`a Theorem, up to pass to a subsequence, there exists a $(1,1)$-tensor field $\hat{J}^{(\infty)}$ on $\overline{B_{\st}(0,\pi-\d)}$ of class $\cC^{k+1,\a}$ such that
$$
\Big\|\big((\diff \phi_{\d}^{(n)})^{-1} \circ \hat{J}^{(n)} \circ (\diff \phi_{\d}^{(n)})\big)^i_j -\big(\hat{J}^{(\infty)}\big)^i_j\Big\|_{\cC^{k+1,\a}(\overline{B_{\st}(0,\pi-\d)})} \to 0 \quad \text{ as } n \to +\infty \,\, , \,\, \text{ for any } 1\leq i, j \leq 2m \,\,,
$$
(see e.g. the proof of \cite[Corollary 3.15]{ricci-flow}). By letting $\d \to 0^+$ and using a Cantor diagonal argument, we obtain a well defined limit tensor field $\hat{J}^{(\infty)}$ on the whole ball $B^{2m}$.
Since $(\hat{g}^{(n)}, \hat{J}^{(n)})$ is an almost-Hermitian structure on $B^{2m}$ for any $n \in \bN$, it follows that $(\hat{g}^{(\infty)}, \hat{J}^{(\infty)})$ is an almost-Hermitian structure on $B^{2m}$. In virtue of Lemma \ref{lemma_unifext2}, one can mimic the proof of \cite[Theorem 2.6]{bohm-lafuente-simon} and show that $(B^{2m},\hat{g}^{(\infty)}, \hat{J}^{(\infty)})$ is a locally homogeneous almost-Hermitian space. Finally, in order to prove that the tensor $\hat{J}^{(\infty)}$ is smooth, and hence real-analytic, one can proceed as in the proof of \cite[Theorem B]{pediconi-annsns} by two steps. Firstly, in virtue of Lemma \ref{lemma_unifext2}, one constructs a locally compact and effective local topological group of pseudo-holomorphic isometries acting transitively on $(B^{2m},\hat{g}^{(\infty)}, \hat{J}^{(\infty)})$ around the origin. Then, by the local Myers-Steenrod Theorem \cite[Theorem A]{pediconi-blms}, this turns out to be a transitive local Lie group of pseudo-holomorphic isometries and hence the thesis follows.
\end{proof}

\begin{remark}[Open question, 2] \label{rmk:open-2}
Let us stress that, in contrast with the Riemannian case, this is not a compactness theorem (compare with \cite[Theorem B]{pediconi-annsns}). In fact, even though the limit space is real-analytic, we do not have control on the top order covariant derivative of the limit almost-complex structure, even for $k=0$.
We ask whether it is possible to refine Definition \ref{def:GM} and get new estimates in order to obtain a compactness result.
\end{remark}

\medskip
\section{The space of locally homogeneous almost-Hermitian spaces} \label{sec:space-locally-homog}
\setcounter{equation} 0

\subsection{A parametrization for locally homogeneous almost-Hermitian spaces} \hfill \par

For any $m,q \in \bZ$ with $m\geq1$ and $0\leq q \leq m^2$, we indicate with $\eH^{\rm loc, \rm alm\text{-}\bC}_{q,m}$ the moduli space of unitary transitive Lie algebras of rank $(q,m)$ up to isomorphisms and we indicate with $\eH^{\rm alm\text{-}\bC}_{q,m}$ the subset of moduli space of regular ones. Similarly, with $\eH^{\rm loc, \bC}_{q,m}$ (resp. $\eH^{\bC}_{q,m}$) denotes the subset of integrable unitary transitive Lie algebras (resp. the regular ones). We fix a decomposition $\bR^{q+2m}=\bR^q\oplus\bR^{2m}$ and the corresponding diagonal embedding of $\fGL(q) \times \fU(m)$ into $\fGL(q{+}2m)$. Accordingly, we denote by $\mathrm{pr}_{\bR^{2m}} : \bR^{q+2m} \to \bR^{2m}$ the induced natural projection onto the second factor. We define
$$
\cW_{q,m} \= \big(\fGL(q) \times \fU(m)\big) \big\backslash \big(\L^2(\bR^{q+2m})^*\otimes \bR^{q+2m}\big) \,\, ,
$$ where $\fGL(q) \times \fU(m)$ acts on $\L^2(\bR^{q+2m})^*\otimes \bR^{q+2m}$ on the left by change of basis. Following \cite{lauret-jlms}, one can prove that the map
$$
\Psi_{q,m}: \eH^{\rm loc, \rm alm\text{-}\bC}_{q,m} \to \cW_{q,m} \,\, , \quad (\gg=\gh+\gm,I,\langle\,,\rangle) \mapsto \mu \= u^*\big([\cdot,\cdot]_{\gg}\big) \,\, ,
$$
where $u: \bR^{q+2m} \to \gg$ is any adapted linear frame for $(\gg=\gh+\gm,I,\langle\,,\rangle)$, is well defined, injective and that its image contains precisely the elements $\mu \in \cW_{q,m}$ which verify the following conditions:
\begin{itemize}[leftmargin=30pt]
\item[(h1)] $\mu$ satisfies the Jacobi condition and $\mu(\bR^q,\bR^q) \subset \bR^q$, $\mu(\bR^q,\bR^{2m}) \subset \bR^{2m}$;
\item[(h2)] $\langle\mu(Z,X),Y\rangle_{\st}=\langle X,\mu(Z,Y)\rangle_{\st}$ for any $X,Y \in \bR^{2m}$, $Z \in \bR^q$;
\item[(h3)] $\mu(Z,I_{\st}X) = I_{\st}\mu(Z,X)$ for any $X \in \bR^{2m}$, $Z \in \bR^q$;
\item[(h4)] $\big\{Z\in \bR^q : \mu(Z,\bR^{2m})=\{0\}\big\}=\{0\}$.
\end{itemize}
The image of $\eH^{\rm loc, \bC}_{q,m}$ is characterized by the further condition
\begin{itemize}[leftmargin=30pt]
\item[(h5)] $\mathrm{pr}_{\bR^{2m}} \big(\mu(I_{\st}X,I_{\st}Y)-\mu(X,Y)\big) = I_{\st} \mathrm{pr}_{\bR^{2m}} \big(\mu(I_{\st}X,Y)+\mu(X,I_{\st}Y)\big)$ for any $X,Y \in \bR^{2m}$.
\end{itemize}

\begin{remark} \label{rmk:alm-eff}
We point out that, while conditions (h1), (h2), (h3), (h5) are closed, condition (h4) is open. However, following \cite{lauret-jlms}, the following fact holds: for any element $\tilde{\mu} \in \cW_{q,m} \setminus \eH^{\rm loc, \rm alm\text{-}\bC}_{q,m}$ satisfying conditions (h1), (h2) and (h3), there exist a unique integer $0 \leq q' < q$ and a decomposition $\bR^q= \bR^{q-q'}\oplus\bR^{q'}$ such that $\bR^{q-q'}=\big\{Z\in \bR^q : \mu(Z,\bR^{2m})=\{0\}\big\}$ and
$$
(\tilde{\mu})_{|q',m} \= {\rm{pr}}_{\bR^{q'+2m}} \circ (\tilde{\mu}|_{\bR^{q'+2m}\times\bR^{q'+2m}}) \in \eH^{\rm loc, \rm alm\text{-}\bC}_{q',m} \,\, ,
$$
where $\bR^{q'+2m}=\bR^{q'} \oplus \bR^{2m}$ and ${\rm{pr}}_{\bR^{q'+2m}}: \bR^{q+2m} \to \bR^{q'+2m}$ is the projection with respect to the decomposition $\bR^{q+2m}=\bR^{q-q'} \oplus \bR^{q'+2m}$. \end{remark}

From now on, we identify $\eH^{\rm loc, \rm alm\text{-}\bC}_{q,m}$ with its image through $\Psi_{q,m}$ and, for any $\mu \in \eH^{\rm loc, \rm alm\text{-}\bC}_{q,m} \simeq \Psi_{q,m}(\eH^{\rm loc, \rm alm\text{-}\bC}_{q,m})$, we set $$\gg_{\mu}\=(\bR^{q+2m},\mu) \,\, , \quad \gh_{\mu}\=(\bR^{q},\mu|_{\bR^{q}\times\bR^{q}})$$ so that $(\gg_{\mu}=\gh_{\mu}+\bR^{2m},I_{\st},\langle\,,\rangle_{\st})$ is the unitary transitive Lie algebra uniquely associated to the bracket $\mu$. We also set
\begin{equation} \label{unionq}
\eH^{\rm loc, \rm alm\text{-}\bC}_{m} \= \bigcup_{q=0}^{m^2} \eH^{\rm loc, \rm alm\text{-}\bC}_{q,m} \,\, , \quad \eH^{\rm alm\text{-}\bC}_{m} \= \bigcup_{q=0}^{m^2} \eH^{\rm alm\text{-}\bC}_{q,m} \,\, , \quad
\eH^{\rm loc, \bC}_{m} \= \bigcup_{q=0}^{m^2} \eH^{\rm loc, \bC}_{q,m} \,\, , \quad \eH^{\bC}_{m} \= \bigcup_{q=0}^{m^2} \eH^{\bC}_{q,m} \,\, .
\end{equation}

The set $\eH^{\rm loc, \rm alm\text{-}\bC}_{q,m}$ parametrizes the moduli space of the equivalence classes of $m$-dimensional locally homogeneous almost-Hermitian spaces, up to local equivariant pseudo-holomorphic isometries, in the following way.

\begin{theorem} \label{thm:Sp}
For any unitary transitive Lie algebra $\mu \in \eH^{\rm loc, \rm alm\text{-}\bC}_m$, there exist a pointed locally homogeneous almost-Hermitian space $((M,J,g),p)$ and an injective homomorphism $\f: \gg_{\mu} \to \gkill^{J,g}$ such that 
$$
\f(\gh_{\mu}) \subset \gkill^{J,g}_0 \,\, , \quad \f(\bR^{2m}) = \gm^g \,\, , \quad \f|_{\bR^{2m}} \circ I_{\st} \circ (\f|_{\bR^{2m}}){}^{-1} = I_J \,\, , \quad ((\f|_{\bR^{2m}}){}^{-1})^*\langle\,,\rangle_{\st} = \langle\,,\rangle_{g} \,\, ,
$$
where $(\gkill^{J,g}=\gkill^{J,g}_0+\gm^g,I_J,\langle\,,\rangle_g)$ is the Hermitian Nomizu algebra of $(M,J,g)$ at $p$, as in \eqref{eq:dec-nomizu}. The space $(M,J,g)$ is uniquely determined up to a local equivariant pseudo-holomorphic isometry.
Moreover, $J$ is integrable if and only if $\mu \in \eH^{\rm loc, \bC}_m$, and $(M,J,g)$ is locally equivariantly pseudo-holomorphically isometric to a globally homogeneous almost-Hermitian space if and only if $\mu$ is regular.
\end{theorem}

\begin{proof}
The analogue statement in the category of locally homogeneous Riemannian spaces follows from \cite[Lemma 3.5 and Prop 4.4]{spiro}. Here, we just sketch the construction of the pointed locally homogeneous almost-Hermitian space associated to an element $\mu \in \eH^{\rm loc,\rm alm\text{-}\bC}_m$. Let $\fG_{\mu}$ be the unique simply connected Lie group with $\Lie(\fG_{\mu})=\gg_{\mu}$ and $\fH_{\mu} \subset \fG_{\mu}$ the connected Lie subgroup with $\Lie(\fH_{\mu})=\gh_{\mu}$, which is closed in $\fG_{\mu}$ if and only if $\mu$ is regular. Then one can consider the {\it local quotient} of Lie groups ${\fG}_{\mu}/{\fH}_{\mu}$, which admits a unique suitable real-analytic manifold structure (see e.g. \cite[Sect 6]{pediconi-blms}). Moreover, by means of the standard local action of ${\fG}_{\mu}$ on ${\fG}_{\mu}/{\fH}_{\mu}$, one can construct a uniquely determined invariant almost-Hermitian structure $(J_{\mu},g_{\mu})$ on ${\fG}_{\mu}/{\fH}_{\mu}$ such that $(\bR^{2m},I_{\st},\langle\,,\rangle_{\st}) \simeq (T_{e_{\mu}{\fH}_{\mu}}{\fG}_{\mu}/{\fH}_{\mu},J_{\mu}|_{e_{\mu}{\fH}_{\mu}},g_{\mu}|_{e_{\mu}{\fH}_{\mu}})$.
\end{proof}

\subsection{Gauduchon connections of locally homogeneous almost-Hermitian spaces} \label{Curvmu} \hfill \par

For any $\mu \in \eH^{\rm loc, \rm alm\text{-}\bC}_{q,m}$, we will refer to all the geometric data of $(\fG_{\mu}/\fH_{\mu}, J_{\mu}, g_{\mu})$ by writing $\mu$, for example, $D^\mu$ will denote the Levi-Civita connection and $\widetilde{D}^{\mu}$ will denote its Hermitian Ambrose-Singer connection, uniquely determined by the fixed reductive decomposition $\gg_\mu=\gh_\mu+\bR^{2m}$. Moreover, we consider the orthogonal decomposition
\begin{equation} \label{eq:decmu}
\mu = (\mu|_{\gh_{\mu} \wedge \gg_{\mu}})+\mu_{\gh_{\mu}}+\mu_{\bR^{2m}} \,\, , \quad \text{ where } \quad \mu_{\gh_{\mu}}: \bR^{2m} \wedge \bR^{2m} \to \gh_{\mu} \,\, , \quad \mu_{\bR^{2m}}: \bR^{2m} \wedge \bR^{2m} \to \bR^{2m} \,\, ,
\end{equation}
with respect to the the $\ad(\gh_{\mu})$-invariant product $\langle\,,\rangle'_{\mu}$ on $\gg_{\mu}$ introduced in Remark \ref{rem:Sp}. We denote by $F^{\mu} \in \L^3(\bR^{2m})^*$ the three-form corresponding to $\diff^{\,c}\w_{\mu}$, which is given by
$$
F^{\mu}(X,Y,Z) = -\langle \mu_{\bR^{2m}}(I_{\st}X,I_{\st}Y), Z \rangle_{\st} -\langle \mu_{\bR^{2m}}(I_{\st}Y,I_{\st}Z), X \rangle_{\st} -\langle \mu_{\bR^{2m}}(I_{\st}Z,I_{\st}X), Y \rangle_{\st} \,\, .
$$
We also denote by $N^\mu$ the Nijenhuis tensor, which is given by
$$
N^\mu(X,Y) = -\mu_{\bR^{2m}}(I_{\st}X,I_{\st}Y) +\mu_{\bR^{2m}}(X,Y) +I_{\st}\mu_{\bR^{2m}}(I_{\st}X,Y) +I_{\st}\mu_{\bR^{2m}}(X,I_{\st}Y) \,\, .
$$
According to \eqref{eq:decomposition}, we consider the decomposition $F^\mu=(F^{\mu})^++(F^\mu)^-$.
By \cite[Eqns (1.2.1) and (2.2.4)]{gauduchon-bumi}, we have
$$\begin{aligned}
(F^\mu)^-(X,Y,Z) &= \langle N^\mu(X,Y), Z \rangle_{\st} + \langle N^\mu(Y,Z), X \rangle_{\st} + \langle N^\mu(Z,X), Y \rangle_{\st} \,\,,\\
(F^\mu)^+(X,Y,Z) &= -\langle \mu_{\bR^{2m}}(X,Y), Z \rangle_{\st}-\langle \mu_{\bR^{2m}}(Y,Z), X \rangle_{\st}-\langle \mu_{\bR^{2m}}(Z,X), Y \rangle_{\st} \\
&\quad +\langle \mu_{\bR^{2m}}(I_{\st}X,Y), I_{\st}Z \rangle_{\st}+\langle \mu_{\bR^{2m}}(I_{\st}Y,Z), I_{\st}X \rangle_{\st}+\langle \mu_{\bR^{2m}}(I_{\st}Z,X), I_{\st}Y \rangle_{\st}\\
&\quad +\langle \mu_{\bR^{2m}}(X,I_{\st}Y), I_{\st}Z \rangle_{\st}+\langle \mu_{\bR^{2m}}(Y,I_{\st}Z), I_{\st}X \rangle_{\st}+\langle \mu_{\bR^{2m}}(Z,I_{\st}X), I_{\st}Y \rangle_{\st}\,\,.
\end{aligned}$$

Fix a parameter $t\in\bR$ and look at the Gauduchon connection $\n^{t,\mu}$. Let us consider now the $(1,2)$-tensors
$$
S^{\mu} \= \widetilde{D}^{\mu} -D^{\mu} \,\, , \quad A^{t,{\mu}}\= \widetilde{D}^{\mu} -\n^{t,{\mu}} \,\,,
$$
which can be identified with linear maps
$$
S^{\mu}: \bR^{2m} \to \so(2m) \,\, , \quad A^{t,{\mu}}: \bR^{2m} \to \gu(m) \,\,.
$$
For the operator $S^\mu$, by \cite[Ch X, Thm 3.3]{kobayashi-nomizu-2}), we have
\begin{equation} \label{eq:defS}
\langle S^{\mu}(X)Y,Z\rangle_{\st} = -\tfrac12\langle\mu_{\bR^{2m}}(X,Y),Z\rangle_{\st} -\tfrac12\langle\mu_{\bR^{2m}}(Z,X),Y\rangle_{\st} -\tfrac12\langle\mu_{\bR^{2m}}(Z,Y),X\rangle_{\st} \,\, .
\end{equation}
For the operator $A^{t,\mu}$, by \eqref{eq:gauduchon-almost}, we have
\begin{multline*}
\langle A^{t,\mu}(X)Y,Z \rangle_{\st} = \langle S^\mu(X)Y,Z\rangle_{\st} +\tfrac{t+1}4(F^\mu)^+(X,I_{\st} Y,I_{\st} Z) +\tfrac{t-1}4(F^\mu)^+(X,Y,Z) \\
+\tfrac14 \langle N^\mu(Y,Z), X \rangle_{\st} +\tfrac12(F^\mu)^-(X,Y,Z) \,\, .
\end{multline*}
Then, by \cite[Thm 2.3, Ch X]{kobayashi-nomizu-2}, the Riemannian curvature is explicitly given by
\begin{equation} \label{eq:Riemcurvature}
\Rm(\mu)(X,Y) = \ad_{\mu}\big(\mu_{\gh_{\mu}}(X,Y)\big)\vert_{\bR^{2m}} -[S^{\mu}(X), S^{\mu}(Y)] -S^{\mu}(\mu_{\bR^{2m}}(X,Y)) \,\, , \\
\end{equation}
and, analogously, the $t$-Gauduchon curvature and torsion are given by
\begin{equation} \begin{gathered} \label{eq:t-curtor}
T^t(\mu)(X,Y) = A^{t,{\mu}}(X)Y -A^{t,{\mu}}(Y)X -\mu_{\bR^{2m}}(X,Y) \,\, , \\
\W^t(\mu)(X,Y) = \ad_{\mu}\big(\mu_{\gh_{\mu}}(X,Y)\big)\vert_{\bR^{2m}} -[A^{t,{\mu}}(X), A^{t,{\mu}}(Y)] -A^{t,{\mu}}(\mu_{\bR^{2m}}(X,Y)) \,\, . \\
\end{gathered} \end{equation}

Moreover, we recall that any $\fG_\mu$-invariant tensor field $Q$ on $\fG_\mu / \fH_\mu$ is parallel with respect to $\tilde D^\mu$, see e.g. the proof of \cite[Prop 2.7, Ch X]{kobayashi-nomizu-2}.
Therefore, for the covariant derivatives $D^\mu Q$ and $\n^{t,\mu}Q$, we have
\begin{equation} \label{eq:covder}
X\lrcorner D^\mu Q = -S^\mu(X) \cdot Q \,\,,
\qquad
X\lrcorner \n^{t,\mu} Q = -A^{t,\mu}(X) \cdot Q \,\,.
\end{equation}

\subsection{A potpourri of topologies in the moduli space}\label{sec:potpourri} \hfill \par

We are going to introduce some topologies on the moduli space $\eH^{\rm loc, \rm alm\text{-}\bC}_{m}$. The first one is the so-called {\it algebraic convergence}, that is

\begin{definition} \label{def:algconv}
A sequence $(\mu^{(n)}) \subset \eH^{\rm loc, \rm alm\text{-}\bC}_{q,m}$ is said to {\it converge algebraically to $\mu^{(\infty)} \in \eH^{\rm loc, \rm alm\text{-}\bC}_{m}$} if one of the following conditions is satisfied:
\begin{itemize}
\item[i)] $\mu^{(\infty)} \in \eH^{\rm loc, \rm alm\text{-}\bC}_{q,m}$ and $\mu^{(n)} \to \mu^{(\infty)}$ in the standard topology induced by $\cW_{q,m}$;
\item[ii)] $\mu^{(\infty)} \in \eH^{\rm loc, \rm alm\text{-}\bC}_{q',m}$ for some $0\leq q' <q$ and there exists $\tilde{\mu}^{(\infty)} \in \cW_{q,m} \setminus \eH^{\rm loc, \rm alm\text{-}\bC}_{q,m}$ such that $\mu^{(n)} \to \tilde{\mu}^{(\infty)}$ in the standard topology of $\cW_{q,m}$ and $(\tilde{\mu}^{(\infty)})_{|q',m}=\mu^{(\infty)}$ as in Remark \ref{rmk:alm-eff}.
\end{itemize}
\end{definition}

For the second notion of convergence, we notice that Theorem \ref{thm:CNcurvmod} and Theorem \ref{thm:Sp} give rise to a well defined map 
$$\eH^{\rm loc, \rm alm\text{-}\bC}_{m} \to \cX^s(m) \,\, , \quad \mu \mapsto \theta^s(\mu)$$
that assigns to any $\mu \in \eH^{\rm loc, \rm alm\text{-}\bC}_{m}$ the corresponding Hermitian $s$-tuples $\theta^s(\mu)$ of $(\fG_{\mu}/\fH_{\mu}, J_{\mu}, g_{\mu})$, for any $s \geq \jmath(m)+2$ (see Subsection \ref{sec:h-singer}). Let us notice that this map is surjective but not injective. In fact, it holds that $\theta^s(\mu_1)=\theta^s(\mu_2)$ for some, and hence for any, $s \geq \jmath(m)+2$ if and only if $\gkill(\mu_1)=\gkill(\mu_2)$. Then, the so-called {\it infinitesimal convergence} is defined as follows.

\begin{definition} \label{def:infconv}
A sequence $(\mu^{(n)}) \subset \eH^{\rm loc, \rm alm\text{-}\bC}_{m}$ is said to {\it converge $s$-infinitesimally to $\mu^{(\infty)} \in \eH^{\rm loc, \rm alm\text{-}\bC}_{m}$}, for some $s \geq \jmath(m){+}2$, if $\theta^s(\mu^{(n)}) \to \theta^s(\mu^{(\infty)})$ as $n \to +\infty$ in the standard topology of $\cX^s(m)$. If $(\mu^{(n)})$ converges $s$-infinitesimally to $\mu^{(\infty)}$ for any $s \geq \jmath(m)+2$, then $(\mu^{(n)})$ is said to {\it converge infinitesimally to $\mu^{(\infty)}$}. 
\end{definition}

By the previous observation, uniqueness of limit has to be intended in the following way: if a sequence $(\mu^{(n)}) \subset \eH^{\rm loc, \rm alm\text{-}\bC}_{m}$ converges $s_1$-infinitesimally to $\mu^{(\infty)}_1$ and $s_2$-infinitesimally to $\mu^{(\infty)}_2$ for some integers $s_2 \geq s_1 \geq \jmath(m)+2$, then $\gkill(\mu^{(\infty)}_1)=\gkill(\mu^{(\infty)}_2)$. We also mention that our notion of infinitesimal convergence is equivalent to the original notion introduced by Lauret in \cite[Sect 6]{lauret-jlms} and \cite[Sect 3.4]{lauret-rend}. Moreover, since the infinitesimal convergence involves only the germs on the almost-Hermitian structures at the origin, it turns out that it is weaker than the algebraic convergence topology, i.e.

\begin{prop} \label{prop:Lau}
Let $q,m \in \bZ$ with $m \geq 1$ and $0\leq q \leq m^2$. If $(\mu^{(n)}) \subset \eH^{\rm loc, \rm alm\text{-}\bC}_{m}$ converges algebraically to $\mu^{(\infty)} \in \eH^{\rm loc, \rm alm\text{-}\bC}_{m}$, then $(\mu^{(n)})$ converges infinitesimally to $\mu^{(\infty)}$.
\end{prop}

\begin{proof}
Assume that $(\mu^{(n)}) \subset \eH^{\rm loc, \rm alm\text{-}\bC}_{m}$ converges algebraically to $\mu^{(\infty)} \in \eH^{\rm loc, \rm alm\text{-}\bC}_{m}$. From \eqref{eq:defS}, it follows that $S^{\mu^{(n)}} \to S^{\mu^{(\infty)}}$ in the standard Euclidean topology. Therefore, the proof follows from \eqref{eq:Riemcurvature} and \eqref{eq:covder}.
\end{proof}

Notice that, in the Riemannian case, the converse assertion of Proposition \eqref{prop:Lau} does not hold true. A counterexample consisting on a sequence of Ricci flow blow-downs on the universal cover of $\fSL(2,\bR)$ is discussed in \cite[Ex 9.1]{bohm-lafuente}.
The phenomenon of sequences that converge infinitesimally but do not admit any convergent subsequence in the algebraic topology is called {\it algebraic collapse} \cite[Sect 5]{bohm-lafuente}. \smallskip

The last topology we consider in the moduli space $\eH^{\rm loc, \rm alm\text{-}\bC}_{m}$ is the pointed convergence topology (see e.g. Definition \ref{defpointed}). More precisely, by means of Theorem \ref{thm:existence}, for any element $\mu \in \eH^{\rm loc, \rm alm\text{-}\bC}_{m}$ with $|\sec(\mu)| \leq 1$, there exists a unique, up to equivariant pseudo-holomorphic isometry, $2m$-dimensional almost-Hermitian geometric model $(\eB_{\mu},\hat{J}_{\mu},\hat{g}_{\mu})$ in the class $\mu$. For the sake of notation, we set
$$
\eH^{\rm loc, \rm alm\text{-}\bC}_{m}(1) \= \{ \mu \in \eH^{\rm loc, \rm alm\text{-}\bC}_{m} : |\sec(\mu)| \leq 1 \}
$$
and we observe that, for any $\mu \in \eH^{\rm loc, \rm alm\text{-}\bC}_{m}$, there exists a rescaling constant $c >0$ such that $c \cdot \mu \in \eH^{\rm loc, \rm alm\text{-}\bC}_{m}(1)$, where the $\bR_{>0}$-action on the moduli space $\eH^{\rm loc, \rm alm\text{-}\bC}_{m}$, according to the decomposition \eqref{eq:decmu}, is given by
$$
(c \cdot \mu)|_{\gh_{\mu} \wedge \gg_{\mu}} \= \mu|_{\gh_{\mu} \wedge \gg_{\mu}} \,\, , \quad
(c \cdot \mu)_{\gh_{\mu}} \= \tfrac1{c^2}\mu_{\gh_{\mu}} \,\, , \quad
(c \cdot \mu)_{\bR^{2m}} \= \tfrac1c \mu_{\bR^{2m}} \,\, .
$$
Indeed, the space $({\fG}_{c \cdot \mu}/{\fH}_{c \cdot \mu},J_{c \cdot \mu},g_{c \cdot \mu})$ turns out to be locally equivariantly pseudo-holomorphically isometric to $({\fG}_{\mu}/{\fH}_{\mu},J_{\mu},c^2g_{\mu})$, and so $\sec(c \cdot \mu) = \tfrac1{c}\sec(\mu)$. 

Let us notice now that, by the very definition, the convergence in the pointed $\cC^{s+2}$-topology of a sequence of geometric models in $\eH^{\rm loc, \rm alm\text{-}\bC}_{m}(1)$ implies the $s$-infinitesimal convergence. Concerning the opposite implication, the following weaker version holds true.

\begin{theorem} \label{thm:inf2point}
If a sequence $(\mu^{(n)}) \subset \eH^{\rm loc, \rm alm\text{-}\bC}_{m}(1)$ converges $(s+1)$-infinitesimally to $\mu^{(\infty)}\in \eH^{\rm loc, \rm alm\text{-}\bC}_{m}(1)$ for some integer $s \geq \jmath(m)+2$, then the corresponding geometric models $(\eB_{\mu^{(n)}},\hat{J}_{\mu^{(n)}},\hat{g}_{\mu^{(n)}})$ converge to the geometric model $(\eB_{\mu^{(\infty)}},\hat{J}_{\mu^{(\infty)}},\hat{g}_{\mu^{(\infty)}})$ in the pointed $\cC^{s+2,\a}$-topology for any $0\leq\a<1$.
\end{theorem}

\begin{proof}
Assume that $(\mu^{(n)}) \subset \eH^{\rm loc, \rm alm\text{-}\bC}_{m}(1)$ converges $(s+1)$-infinitesimally to $\mu^{(\infty)} \in \eH^{\rm loc, \rm alm\text{-}\bC}_{m}(1)$ for some integer $s \geq \jmath(m)+2$. Then, by Proposition \ref{prop:stime} and Theorem \ref{thm:conv}, one can pass to a subsequence $(\mu^{(n_i)}) \subset (\mu^{(n)})$ such that the associated almost-Hermitian geometric models $(\eB_{\mu^{(n_i)}},\hat{J}_{\mu^{(n_i)}},\hat{g}_{\mu^{(n_i)}})$ converge to a limit geometric model in the pointed $\cC^{s+2,\a}$-topology for any $0\leq\a<1$ as $i \to +\infty$. By Theorem \ref{thm:CNcurvmod}, any convergent subsequence of $(\eB_{\mu^{(n)}},\hat{J}_{\mu^{(n)}},\hat{g}_{\mu^{(n)}})$ in the pointed $\cC^{s+2,\a}$-topology necessarily converges to the almost-Hermitian geometric model $(\eB_{\mu^{(\infty)}},\hat{J}_{\mu^{(\infty)}},\hat{g}_{\mu^{(\infty)}})$ of $\mu^{(\infty)}$. This implies that the full sequence $(\eB_{\mu^{(n)}},\hat{J}_{\mu^{(n)}},\hat{g}_{\mu^{(n)}})$ converges in the pointed $\cC^{s+2,\a}$-topology to $(\eB_{\mu^{(\infty)}},\hat{J}_{\mu^{(\infty)}},\hat{g}_{\mu^{(\infty)}})$.
\end{proof}

As a direct corollary, we obtain

\begin{corollary}\label{cor:inf-pointed}
A sequence $(\mu^{(n)}) \subset \eH^{\rm loc, \rm alm\text{-}\bC}_{m}(1)$ converges infinitesimally to $\mu^{(\infty)} \in \eH^{\rm loc, \rm alm\text{-}\bC}_{m}(1)$ if and only if the corresponding geometric models $(\eB_{\mu^{(n)}},\hat{J}_{\mu^{(n)}},\hat{g}_{\mu^{(n)}})$ converge to the geometric model $(\eB_{\mu^{(\infty)}},\hat{J}_{\mu^{(\infty)}},\hat{g}_{\mu^{(\infty)}})$ in the pointed $\cC^{\infty}$-topology.
\end{corollary}

We end this section by summarizing the various topologies in the following diagram
$$
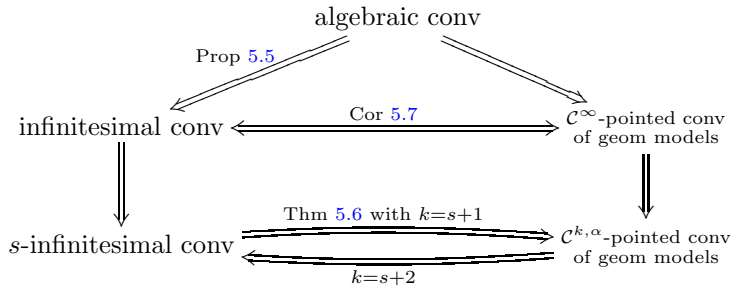
\xymatrix{
& \text{algebraic conv} \ar@{=>}[dl]_{\text{Prop \ref{prop:Lau}}\,\,} \ar@{=>}[dr] & \\
\text{infinitesimal conv} \ar@{=>}[d] \ar@{<=>}[rr]^{\text{Cor \ref{cor:inf-pointed}}} & & {\substack{\cC^{\infty}\text{-pointed conv}\\\text{of geom models}}} \ar@{=>}[d] \\
s\text{-infinitesimal conv} \ar@{=>}@/^/[rr]^{\text{Thm \ref{thm:inf2point} with }k=s+1} & & {\substack{\cC^{k,\a}\text{-pointed conv}\\\text{of geom models}}} \ar@{=>}@/^/[ll]^{k=s+2}
}
$$
and by collecting some open problems.
\begin{remark}[Open questions, 3] \label{rmk:open-3} \hfill \par
\begin{itemize}
\item[i)] Find an explicit example of algebraic collapse for locally homogeneous almost-Hermitian spaces.
\item[ii)] Show that, in real dimension $2m>2$, the $s$-infinitesimal convergence is strictly weaker than the $(s+1)$-infinitesimal convergence for any $s \geq \jmath(m)+2$. In the Riemannian case, this has been proven in \cite[Theorem C]{pediconi-geomded} by using a slight modification of Berger spheres.
\item[iii)] We do not know whether the $s$-infinitesimal convergence is equivalent to the convergence of geometric models in the pointed $\cC^{s+2}$-topology. In contrast to the other questions, this is open also in the Riemannian case.
\end{itemize}
\end{remark}

\section{Examples of explicit computations of Gauduchon curvatures}\label{sec:examples}
\setcounter{equation} 0

\subsection{The Iwasawa threefold} \label{sec:iwasawa} \hfill \par

Consider the {\it Iwasawa manifold}
$$ M \= {\rm Heis}(3;\bZ[\ti])\backslash {\rm Heis}(3;\bC) \,\,,$$
namely, the compact $3$-dimensional complex manifold defined as the quotient of the $3$-dimensional complex Heisenberg group
$$
{\rm Heis}(3;\bC) \= \left\{ \left(\begin{matrix}
1 & z^1 & z^3 \\
& 1 & z^2 \\
& & 1
\end{matrix}\right)\in\GL(3;\bC) \;:\; z^1,z^2,z^3 \in \bC \right\}
$$
by the cocompact discrete subgroup ${\rm Heis}(3;\bZ[\ti])\={\rm Heis}(3;\bC)\cap \GL(3;\bZ[\ti])$. Denote by $J$ the complex structure of $M$ induced by the natural left-invariant complex structure of ${\rm Heis}(3;\bC)$.

Let $g$ be a locally homogeneous Hermitian structure on $(M,J)$. Fix $(e_0,\ldots,e_5)$ unitary frame for $TM$ with respect to $(J,g)$, and denote by $(e^0,\ldots,e^5)$ the dual coframe for $T^*M$. The general structure equations are
\begin{equation} \label{eq:structure-iwasawa}
\mu(e_0,e_2)=\a e_4 \,\,,\quad
\mu(e_0,e_3)=\a e_5 \,\,,\quad
\mu(e_1,e_2)=\a e_5 \,\,,\quad
\mu(e_1,e_3)=-\a e_4 \,\,,
\end{equation}
depending on parameters $\a\in\mathbb R^{>0}$. Note that any such Hermitian metric is {\em balanced} in the sense of Michelsohn, namely, $\diff\w^2=0$.

\begin{rem}\label{rmk:std-iwasawa}
Equations \eqref{eq:structure-iwasawa} can be derived as follows.
Consider the standard left-invariant coframe of $(1,0)$-forms on ${\rm Heis}(3;\bC)$ given by
$$
\f^1\=\diff z^1 \,\,, \quad \f^2\=\diff z^2 \,\,,\quad \f^3\=\diff z^3-z^1\diff z^2 \,\, ,
$$
and notice that the structure equations, with respect to this coframe, are
$$
\diff\f^1=0 \,\,, \quad \diff\f^2=0 \,\,, \quad \diff\f^3=-\f^1\wedge\f^2 \,\, .
$$
Then, by \cite[p 1032]{ugarte-villacampa}, the fundamental $(1,1)$-form associated to any left-invariant Hermitian metric $g$ on ${\rm Heis}(3;\bC)$ has the form
$$
2\w = - \ti r^2 \f^1\wedge\bar\f^1 - \ti \sigma^2 \f^2\wedge\bar\f^2 - \ti \tau^2 \f^3\wedge\bar\f^3 + \big( u \f^1\wedge\bar\f^2 - \bar u \f^2\wedge\bar\f^1 \big) \,\, ,
$$
where $r,\sigma,\tau\in\mathbb R^{>0}$ and $u\in\mathbb C$ are such that $r^2\sigma^2>|u|^2$.
By changing frame to make it $(J,g)$-unitary, \eqref{eq:structure-iwasawa} follows by setting
$$
\a=\sqrt{\tfrac{r^2}{r^2\sigma^2-|u|^2}}\cdot\tfrac{\tau}{r}\,\, .
$$
In particular, the standard Hermitian metric corresponds to parameter $\a=1$.
\end{rem}

By using formulas in Section \ref{Curvmu}, we can compute all the relevant geometric data of $(M,J,g)$.
We clearly have that $N^\mu=0$, hence $(F^\mu)^-=0$ and $(F^\mu)^+=F^\mu$. Moreover, $F^\mu$ has the following non-zero components, up to symmetries:
$$
F^\mu(e_0,e_2,e_4)=F^\mu(e_0,e_3,e_5)=F^\mu(e_1,e_2,e_5)=-F^\mu(e_1,e_3,e_4)=-\a \,\,.
$$
It is straighforward to compute
\begin{footnotesize}
$$
\begin{array}{lll}
S^\mu(e_0)=
\tfrac{\a}{2}
\left(\begin{array}{cccccc}
0 & 0 & 0 & 0 & 0 & 0 \\
0 & 0 & 0 & 0 & 0 & 0 \\
0 & 0 & 0 & 0 & 1 & 0 \\
0 & 0 & 0 & 0 & 0 & 1 \\
0 & 0 & -1 & 0 & 0 & 0 \\
0 & 0 & 0 & -1 & 0 & 0
\end{array}\right)
\,,
&
S^\mu(e_1)=
\tfrac{\a}{2}
\left(\begin{array}{cccccc}
0 & 0 & 0 & 0 & 0 & 0 \\
0 & 0 & 0 & 0 & 0 & 0 \\
0 & 0 & 0 & 0 & 0 & 1 \\
0 & 0 & 0 & 0 & -1 & 0 \\
0 & 0 & 0 & 1 & 0 & 0 \\
0 & 0 & -1 & 0 & 0 & 0
\end{array}\right)
\,,
&
S^\mu(e_2)=
\tfrac{\a}{2}
\left(\begin{array}{cccccc}
0 & 0 & 0 & 0 & -1 & 0 \\
0 & 0 & 0 & 0 & 0 & -1 \\
0 & 0 & 0 & 0 & 0 & 0 \\
0 & 0 & 0 & 0 & 0 & 0 \\
1 & 0 & 0 & 0 & 0 & 0 \\
0 & 1 & 0 & 0 & 0 & 0
\end{array}\right)
\,\,,
\\[10pt]
S^\mu(e_3)=
\tfrac{\a}{2}
\left(\begin{array}{cccccc}
0 & 0 & 0 & 0 & 0 & -1 \\
0 & 0 & 0 & 0 & 1 & 0 \\
0 & 0 & 0 & 0 & 0 & 0 \\
0 & 0 & 0 & 0 & 0 & 0 \\
0 & -1 & 0 & 0 & 0 & 0 \\
1 & 0 & 0 & 0 & 0 & 0
\end{array}\right)
\,,
&
S^\mu(e_4)=
\tfrac{\a}{2}
\left(\begin{array}{cccccc}
0 & 0 & -1 & 0 & 0 & 0 \\
0 & 0 & 0 & 1 & 0 & 0 \\
1 & 0 & 0 & 0 & 0 & 0 \\
0 & -1 & 0 & 0 & 0 & 0 \\
0 & 0 & 0 & 0 & 0 & 0 \\
0 & 0 & 0 & 0 & 0 & 0
\end{array}\right)
\,,
&
S^\mu(e_5)=
\tfrac{\a}{2}
\left(\begin{array}{cccccc}
0 & 0 & 0 & -1 & 0 & 0 \\
0 & 0 & -1 & 0 & 0 & 0 \\
0 & 1 & 0 & 0 & 0 & 0 \\
1 & 0 & 0 & 0 & 0 & 0 \\
0 & 0 & 0 & 0 & 0 & 0 \\
0 & 0 & 0 & 0 & 0 & 0
\end{array}\right)
\,,
\end{array}
$$
\end{footnotesize}
and
\begin{footnotesize}
$$
\begin{gathered}
A^{t,\mu}(e_0)=
\tfrac{\a (t - 1)}{2}
\left(\begin{array}{cccccc}
0 & 0 & 0 & 0 & 0 & 0 \\
0 & 0 & 0 & 0 & 0 & 0 \\
0 & 0 & 0 & 0 & -1 & 0 \\
0 & 0 & 0 & 0 & 0 & -1 \\
0 & 0 & 1 & 0 & 0 & 0 \\
0 & 0 & 0 & 1 & 0 & 0
\end{array}\right)
\,\,,
\qquad
A^{t,\mu}(e_1)=
\tfrac{\alpha (t - 1)}{2}
\left(\begin{array}{cccccc}
0 & 0 & 0 & 0 & 0 & 0 \\
0 & 0 & 0 & 0 & 0 & 0 \\
0 & 0 & 0 & 0 & 0 & -1 \\
0 & 0 & 0 & 0 & 1 & 0 \\
0 & 0 & 0 & -1 & 0 & 0 \\
0 & 0 & 1 & 0 & 0 & 0
\end{array}\right)
\,\,,
\\[10pt]
A^{t,\mu}(e_2)=
\tfrac{\alpha (t - 1)}{2}
\left(\begin{array}{cccccc}
0 & 0 & 0 & 0 & 1 & 0 \\
0 & 0 & 0 & 0 & 0 & 1 \\
0 & 0 & 0 & 0 & 0 & 0 \\
0 & 0 & 0 & 0 & 0 & 0 \\
-1 & 0 & 0 & 0 & 0 & 0 \\
0 & -1 & 0 & 0 & 0 & 0
\end{array}\right)
\,\,,
\qquad
A^{t,\mu}(e_3)=
\tfrac{\alpha (t - 1)}{2}
\left(\begin{array}{cccccc}
0 & 0 & 0 & 0 & 0 & 1 \\
0 & 0 & 0 & 0 & -1 & 0 \\
0 & 0 & 0 & 0 & 0 & 0 \\
0 & 0 & 0 & 0 & 0 & 0 \\
0 & 1 & 0 & 0 & 0 & 0 \\
-1 & 0 & 0 & 0 & 0 & 0
\end{array}\right)
\,\,,
\\[10pt]
A^{t,\mu}(e_4)=
A^{t,\mu}(e_5)=0
\,\,.
\end{gathered}
$$
\end{footnotesize}
In particular, it is easy to check that the Chern connection (corresponding to parameter $t=1$) is flat.
Finally, we compute the Gauduchon curvature: the non-zero components are
\begin{tiny}
$$
\begin{gathered}
\Omega^{t}(\mu)(e_0,e_1)=
\tfrac{\alpha^{2} (t - 1)^{2}}{2}
\left(\begin{array}{cccccc}
0 & 0 & 0 & 0 & 0 & 0 \\
0 & 0 & 0 & 0 & 0 & 0 \\
0 & 0 & 0 & -1 & 0 & 0 \\
0 & 0 & 1 & 0 & 0 & 0 \\
0 & 0 & 0 & 0 & 0 & 1 \\
0 & 0 & 0 & 0 & -1 & 0
\end{array}\right)
\,\,
\qquad
\Omega^{t}(\mu)(e_0,e_2)=
\tfrac{\alpha^{2} (t - 1)^{2}}{4}
\left(\begin{array}{cccccc}
0 & 0 & 1 & 0 & 0 & 0 \\
0 & 0 & 0 & 1 & 0 & 0 \\
-1 & 0 & 0 & 0 & 0 & 0 \\
0 & -1 & 0 & 0 & 0 & 0 \\
0 & 0 & 0 & 0 & 0 & 0 \\
0 & 0 & 0 & 0 & 0 & 0
\end{array}\right)
\,\,
\\[10pt]
\Omega^{t}(\mu)(e_0,e_3)=
\tfrac{\alpha^{2} (t - 1)^{2}}{4}
\left(\begin{array}{cccccc}
0 & 0 & 0 & 1 & 0 & 0 \\
0 & 0 & -1 & 0 & 0 & 0 \\
0 & 1 & 0 & 0 & 0 & 0 \\
-1 & 0 & 0 & 0 & 0 & 0 \\
0 & 0 & 0 & 0 & 0 & 0 \\
0 & 0 & 0 & 0 & 0 & 0
\end{array}\right)
\,\,
\qquad
\Omega^{t}(\mu)(e_1,e_2)=
\tfrac{\alpha^{2} (t - 1)^{2}}{4}
\left(\begin{array}{cccccc}
0 & 0 & 0 & -1 & 0 & 0 \\
0 & 0 & 1 & 0 & 0 & 0 \\
0 & -1 & 0 & 0 & 0 & 0 \\
1 & 0 & 0 & 0 & 0 & 0 \\
0 & 0 & 0 & 0 & 0 & 0 \\
0 & 0 & 0 & 0 & 0 & 0
\end{array}\right)
\,\,,
\\[10pt]
\Omega^{t}(\mu)(e_1,e_3)=
\tfrac{\alpha^{2} (t - 1)^{2}}{4}
\left(\begin{array}{cccccc}
0 & 0 & 1 & 0 & 0 & 0 \\
0 & 0 & 0 & 1 & 0 & 0 \\
-1 & 0 & 0 & 0 & 0 & 0 \\
0 & -1 & 0 & 0 & 0 & 0 \\
0 & 0 & 0 & 0 & 0 & 0 \\
0 & 0 & 0 & 0 & 0 & 0
\end{array}\right)
\,\,,
\qquad
\Omega^{t}(\mu)(e_2,e_3)=
\tfrac{\alpha^{2} (t - 1)^{2}}{2}
\left(\begin{array}{cccccc}
0 & -1 & 0 & 0 & 0 & 0 \\
1 & 0 & 0 & 0 & 0 & 0 \\
0 & 0 & 0 & 0 & 0 & 0 \\
0 & 0 & 0 & 0 & 0 & 0 \\
0 & 0 & 0 & 0 & 0 & 1 \\
0 & 0 & 0 & 0 & -1 & 0
\end{array}\right)
\,\,,
\end{gathered}
$$
\end{tiny}

The first Gauduchon-Ricci form is zero. The second Gauduchon-Ricci form is
\begin{tiny}
$$
\rho^{t,(2)}(\mu)=\tfrac{\alpha^2(t-1)^2}{2}
\left(\begin{array}{cccccc}
0 & -1 & 0 & 0 & 0 & 0 \\
1 & 0 & 0 & 0 & 0 & 0 \\
0 & 0 & 0 & -1 & 0 & 0 \\
0 & 0 & 1 & 0 & 0 & 0 \\
0 & 0 & 0 & 0 & 0 & 2 \\
0 & 0 & 0 & 0 & -2 & 0
\end{array}\right) \,\,.
$$
\end{tiny}
Finally, the Gauduchon scalar curvature is clearly $0$.
One can also compute the torsion $T^{t}(\mu)$, its covariant derivates $\n^{t,\mu}T^t(\mu)$, as well as the covariant derivatives $\n^{t,\mu}\Omega^{t}(\mu)$, etc.

\subsection{The Kodaira surface} \label{sec:kodaira} \hfill \par

We consider the {\em primary Kodaira surface}, see e.g. \cite{barth-peters-hulek-vandeven}. It is known that it is a compact quotient of the Lie group
\begin{equation}\label{eq:kodaira}
\fG \= \mathrm{Heis}(3;\bR) \times \bR \,\,,
\end{equation}
where $\mathrm{Heis}(3;\bR)$ denotes the $3$-dimensional real Heisenberg group, by means of the co-compact lattice $\Gamma\=\mathrm{Heis}(3;\bZ)\times\bZ$.
The group $\fG$ can be endowed with a left-invariant complex structure $J_0$, that is unique up to linear equivalence, that moves to the quotient $M\=\Gamma\backslash \fG$. The compact complex surface $(M,J)$ has Kodaira dimension $0$, odd first Betti number, and trivial canonical bundle. 

Arguing as in Remark \ref{rmk:std-iwasawa}, any locally homogeneous Hermitian structure $(J_0,g)$ is described by the unitary frame $(e_0,e_1,e_2,e_3)$ with structure equations
$$
\begin{gathered}
\mu(e_0,e_1)=\tfrac{\alpha}{r}  e_{0} - \tfrac{\beta}{r}  e_{1} - \tfrac{v}{r^{2}}  e_{3} \,\,, 
\qquad
\mu(e_0,e_2)=-\tfrac{\alpha^{2}}{v}  e_{0} + \tfrac{\alpha \beta}{v}  e_{1} + \tfrac{\alpha}{r}  e_{3} \,\,,
\\
\mu(e_0,e_3)=-\tfrac{\alpha \beta}{v}  e_{0} + \tfrac{\beta^{2}}{v}  e_{1} + \tfrac{\beta}{r}  e_{3} \,\,,
\qquad
\mu(e_1,e_2)=\tfrac{\alpha \beta}{v}  e_{0} - \tfrac{\beta^{2}}{v}  e_{1} - \tfrac{\beta}{r}  e_{3} \,\,,
\\
\mu(e_1,e_3)=-\tfrac{\alpha^{2}}{v}  e_{0} + \tfrac{\alpha \beta}{v}  e_{1} + \tfrac{\alpha}{r}  e_{3} \,\,,
\qquad
\mu(e_2,e_3)=\tfrac{{(\alpha^{2} + \beta^{2})} \alpha r}{v^{2}}  e_{0} -\tfrac{{(\alpha^{2} + \beta^{2})} \beta r}{v^{2}} e_{1} -\tfrac{\alpha^{2} + \beta^{2}}{v} e_{3} \,\,,
\end{gathered}
$$
depending on parameters $r,v\in\bR^{>0}$, $\a,\b\in \bR$.
In particular, the standard Hermitian structure corresponds to $r=v=1$ and $\alpha=\beta=0$ (see e.g. \cite{angella-dloussky-tomassini}).
With respect to this frame, the $F^\mu$ form is
$$
F^\mu=-\big( \tfrac{\alpha^{2}}{v} + \tfrac{\beta^{2}}{v} + \tfrac{v}{r^{2}} \big) e_{0} \wedge e_{1} \wedge e_{3} + \big( \tfrac{{(\alpha^{2} + \beta^{2})} \alpha r}{v^{2}} + \tfrac{\alpha}{r} \big)  e_{0} \wedge e_{2} \wedge e_{3} - \big( \tfrac{{(\alpha^{2} + \beta^{2})} \beta r}{v^{2}} + \tfrac{\beta}{r} \big)  e_{1} \wedge e_{2} \wedge e_{3} \,\, .
$$
As before, we can explicitly compute the Levi-Civita connection, the Gauduchon connections, and their related geometric quantities, see Appendix \ref{app:kodaira} for the relevant SageMath code.

As an example, the Chern connection ($t=1$) is given by:
\begin{tiny}
$$
\begin{gathered}
A^{t=1,\mu}( e_{0})=\left(\begin{array}{cccc}
0 & -\tfrac{\alpha}{r} & \tfrac{(\alpha^{2} - \beta^{2}) r^{2} - v^{2}}{2 \, r^{2} v} & \tfrac{\alpha \beta}{v} \\
\tfrac{\alpha}{r} & 0 & -\tfrac{\alpha \beta}{v} & \tfrac{(\alpha^{2}  - \beta^{2}) r^{2} - v^{2}}{2 \, r^{2} v} \\
-\tfrac{(\alpha^{2} - \beta^{2}) r^{2} - v^{2}}{2 \, r^{2} v} & \tfrac{\alpha \beta}{v} & 0 & \tfrac{\alpha}{r} \\
-\tfrac{\alpha \beta}{v} & -\tfrac{(\alpha^{2} - \beta^{2}) r^{2} - v^{2}}{2 \, r^{2} v} & -\tfrac{\alpha}{r} & 0
\end{array}\right) \,\,,
\\[10pt]
A^{t=1,\mu}( e_{1})=\left(\begin{array}{cccc}
0 & \tfrac{\beta}{r} & -\tfrac{\alpha \beta}{v} & \tfrac{(\alpha^{2} - \beta^{2}) r^{2} + v^{2}}{2 \, r^{2} v} \\
-\tfrac{\beta}{r} & 0 & -\tfrac{(\alpha^{2} - \beta^{2}) r^{2} + v^{2}}{2 \, r^{2} v} & -\tfrac{\alpha \beta}{v} \\
\tfrac{\alpha \beta}{v} & \tfrac{(\alpha^{2} - \beta^{2}) r^{2} + v^{2}}{2 \, r^{2} v} & 0 & -\tfrac{\beta}{r} \\
-\tfrac{(\alpha^{2} - \beta^{2}) r^{2} + v^{2}}{2 \, r^{2} v} & \tfrac{\alpha \beta}{v} & \tfrac{\beta}{r} & 0
\end{array}\right) \,\, ,
\\[10pt]
A^{t=1,\mu}( e_{2})=\left(\begin{array}{cccc}
0 & 0 & \frac{{\left(\alpha^{2} + \beta^{2}\right)} r^{2} + v^{2}}{2 \, r v^{2}}\beta  & -\frac{{\left(\alpha^{2} + \beta^{2}\right)} r^{2} + v^{2}}{2 \, r v^{2}} \alpha \\
0 & 0 & \frac{{\left(\alpha^{2} + \beta^{2}\right)} r^{2} + v^{2}}{2 \, r v^{2}}\alpha  & \frac{{\left(\alpha^{2} + \beta^{2}\right)} r^{2} + v^{2}}{2 \, r v^{2}}\beta \\
-\frac{{\left(\alpha^{2} + \beta^{2}\right)} r^{2} + v^{2}}{2 \, r v^{2}} \beta & -\frac{{\left(\alpha^{2} + \beta^{2}\right)} r^{2} + v^{2}}{2 \, r v^{2}} \alpha & 0 & 0 \\
\frac{{\left(\alpha^{2} + \beta^{2}\right)} r^{2} + v^{2}}{2 \, r v^{2}} \alpha & -\frac{{\left(\alpha^{2} + \beta^{2}\right)} r^{2} + v^{2}}{2 \, r v^{2}} \beta & 0 & 0
\end{array}\right) \,\, ,
\\[10pt]
A^{t=1,\mu}( e_{3})=\left(\begin{array}{cccc}
0 & -\frac{\alpha^{2} + \beta^{2}}{v} & \frac{{\left(\alpha^{2} + \beta^{2}\right)} r^{2} - v^{2}}{2 \, r v^{2}} \alpha & \frac{{\left(\alpha^{2} + \beta^{2}\right)} r^{2} - v^{2}}{2 \, r v^{2}} \beta \\
\frac{\alpha^{2} + \beta^{2}}{v} & 0 & -\frac{{\left(\alpha^{2} + \beta^{2}\right)} r^{2} - v^{2}}{2 \, r v^{2}} \beta & \frac{{\left(\alpha^{2} + \beta^{2}\right)} r^{2} - v^{2}}{2 \, r v^{2}} \alpha \\
-\frac{{\left(\alpha^{2} + \beta^{2}\right)} r^{2} - v^{2}}{2 \, r v^{2}} \alpha & \frac{{\left(\alpha^{2} + \beta^{2}\right)} r^{2} - v^{2}}{2 \, r v^{2}} \beta & 0 & \frac{\alpha^{2} + \beta^{2}}{v} \\
-\frac{{\left(\alpha^{2} + \beta^{2}\right)} r^{2} - v^{2}}{2 \, r v^{2}} \beta & -\frac{{\left(\alpha^{2} + \beta^{2}\right)} r^{2} - v^{2}}{2 \, r v^{2}} \alpha & -\frac{\alpha^{2} + \beta^{2}}{v} & 0
\end{array}\right) \,\, .
\end{gathered}
$$
\end{tiny}
It is easy to see that the first Chern-Ricci form vanishes, and the second Chern-Ricci form is given by:
\begin{tiny}
$$
\rho^{t=1,(2)}(\mu)
=\left(\begin{array}{cccc}
0 & -\tfrac{{L_1}^{2} {L_2}}{2r^{4} v^{4}} & -\tfrac{{L_1}^{2} \alpha}{r^{3} v^{3}} & -\tfrac{{L_1}^{2} \beta}{r^{3} v^{3}} \\
\tfrac{{L_1}^{2} {L_2}}{2\, r^{4} v^{4}} & 0 & \tfrac{{L_1}^{2} \beta}{r^{3} v^{3}} & -\tfrac{{L_1}^{2} \alpha}{r^{3} v^{3}} \\
\tfrac{{L_1}^{2} \alpha}{r^{3} v^{3}} & -\tfrac{{L_1}^{2} \beta}{r^{3} v^{3}} & 0 & \tfrac{{L_1}^{2} {L_2}}{2\,r^{4} v^{4}} \\
\tfrac{{L_1}^{2} \beta}{r^{3} v^{3}} & \tfrac{{L_1}^{2} \alpha}{r^{3} v^{3}} & -\tfrac{{L_1}^{2} {L_2}}{2\,r^{4} v^{4}} & 0
\end{array}\right)\,\, ,
$$
\end{tiny}
where we put
$$
L_1\=\alpha^{2} r^{2} + \beta^{2} r^{2} + v^{2}
\,\, ,
\qquad
L_2\=\alpha^{2} r^{2} + \beta^{2} r^{2} - v^{2}
\,\, .
$$

Finally, the Gauduchon scalar curvature is given by
$$
\scal^{t}(\mu)=-\tfrac{{(\alpha^{2} r^{2} + \beta^{2} r^{2} + v^{2})}^{3}}{r^{4} v^{4}}{(t - 1)} \,\, ,
$$
which vanishes for the Chern connection.

\subsection{The Kodaira-Thurston almost-complex \texorpdfstring{$4$}{4}-manifold} \label{sec:kodairathurston-almost} \hfill \par

We consider the same differentiable manifold $M=\Gamma\backslash \fG$ as in the previous Section, where $\fG$ is as in \eqref{eq:kodaira}. It is known that $\fG$ admits another left-invariant almost-complex structure $J_1$, which is non-integrable, that induces an almost-K\"ahlerian structure on the quotient $M$, see e.g. \cite{thurston, tralle-oprea}.

In Appendix \ref{app:kodaira-thurston}, we will construct an orthogonal frame $(w_0,w_1,w_2,w_3)$ such that $J_1w_0=w_2$, $J_1w_1=w_3$, whose structure equations depend on $r,\sigma\in\bR^{>0}$ and $u\=x+\ti y\in\bC$ such that $r^2\sigma^2>x^2+y^2$.
We will also compute the Gauduchon curvatures.

\appendix

\section{SageMath code}\label{app:sagemath}
\setcounter{equation} 0

In this Appendix, we collect the SageMath \cite{sagemath} code that we used for the explicit computations of Section \ref{sec:examples}.
The code is available at \url{https://github.com/danieleangella/locally-homogeneous-hermitian.git}.

\subsection{The Iwasawa threefolds (see Section \ref{sec:iwasawa})} \hfill \par

The following SageMath code
%is available at \url{https://cocalc.com/share/public_paths/b250d0c5a5a460b49c232d87157165cff4fc4e2a} and
has been tested on CoCalc:
\begin{footnotesize}
\begin{lstlisting}
sage: version()
SageMath version 9.3, Release Date: 2021-05-09
\end{lstlisting}
\end{footnotesize}

We will make use of the following functions, to simplify matrices and forms depending on parameters:
\begin{footnotesize}
\begin{lstlisting}
sage: def simp_mat(A, dic={}):
        lista=[]
        for b in A.list():
            try:
                lista.append(b.subs(dic).factor())
            except:
                lista.append(b)
        try:
            return(matrix(A.nrows(),A.ncols(),lista))
        except:
            return(A)

sage: def simp_form(phi, dic={}):
        return(sum([phi.interior_product(b).constant_coefficient().subs(dic).factor()*b
                    for b in E.basis()]))
\end{lstlisting}
\end{footnotesize}

We construct the exterior algebra generated by $e_0, \ldots, e_5$, with Lie bracket determined by the structure equations in \eqref{eq:structure-iwasawa}:
\begin{footnotesize}
\begin{lstlisting}
sage: n = 6
sage: E = ExteriorAlgebra(SR, 'e', n)
sage: _ = var("alpha")
sage: d = E.coboundary({
            (0,2): alpha*E.gens()[4],
            (0,3): alpha*E.gens()[5],
            (1,2): alpha*E.gens()[5],
            (1,3): -alpha*E.gens()[4],
        })
\end{lstlisting}
\end{footnotesize}
We save the structure constants in the following dictionary:
\begin{footnotesize}
\begin{lstlisting}
sage: mu = {(a,b): sum([d(c).interior_product(a*b)*c for c in E.gens()])
            for a in E.gens() for b in E.gens()}
\end{lstlisting}
\end{footnotesize}
We also define the following function, to compute the Lie bracket:
\begin{footnotesize}
\begin{lstlisting}
sage: def Lie(x,y):
        return(sum([x.interior_product(a).constant_coefficient()
                    *y.interior_product(b).constant_coefficient()*mu[(a,b)]
                    for a in E.gens() for b in E.gens()]))
\end{lstlisting}
\end{footnotesize}

The almost-complex structure is given as follows:
\begin{footnotesize}
\begin{lstlisting}
sage: j = matrix(2, [0,-1,1,0])
sage: Jmat=block_diagonal_matrix([j for k in range(n/2)])
sage: J = E.lift_morphism(Jmat)
\end{lstlisting}
\end{footnotesize}
It is easy to check that the almost-complex structure is integrable:
\begin{footnotesize}
\begin{lstlisting}
sage: Nij={(a,b): Lie(J(a),J(b))-Lie(a,b)-J(Lie(J(a),b)+Lie(a,J(b)))
           for a in E.gens() for b in E.gens()}
sage: [Nij[(a,b)] for a in E.gens() for b in E.gens() if Nij[(a,b)]!=0]
[]
\end{lstlisting}
\end{footnotesize}

We compute $F^\mu$:
\begin{footnotesize}
\begin{lstlisting}
sage: F = -J(d(sum([E.gens()[2*j]*J(E.gens()[2*j]) for j in range(n/2)])))
sage: F
-alpha*e0*e2*e4 - alpha*e0*e3*e5 - alpha*e1*e2*e5 + alpha*e1*e3*e4
\end{lstlisting}
\end{footnotesize}

We compute the Levi-Civita connection $S^\mu$:
\begin{footnotesize}
\begin{lstlisting}
sage: S = {x: simp_mat(matrix(n,n,[-1/2*mu[(x,y)].interior_product(z).constant_coefficient()
                                   -1/2*mu[(z,x)].interior_product(y).constant_coefficient()
                                   -1/2*mu[(z,y)].interior_product(x).constant_coefficient()
                                   for z in E.gens() for y in E.gens()])) for x in E.gens()}
\end{lstlisting}
\end{footnotesize}
and the Gauduchon connection $A^{t,\mu}$:
\begin{footnotesize}
\begin{lstlisting}
sage: _ = var("t")
sage: A = {x: simp_mat(matrix(n,n,[S[x][E.gens().index(y),E.gens().index(z)]
                                   +(t+1)/4*F.interior_product(x*J(y)*J(z)).constant_coefficient()
                                   +(t-1)/4*F.interior_product(x*y*z).constant_coefficient()
                                   for y in E.gens() for z in E.gens()])) for x in E.gens()}
\end{lstlisting}
\end{footnotesize}
The Chern connection can be obtained by setting $t=1$:
\begin{footnotesize}
\begin{lstlisting}
sage: ACh = {x: simp_mat(A[x], {t:1}) for x in A.keys()}
\end{lstlisting}
\end{footnotesize}
For example, we can print the \LaTeX\ code for $S^\mu$ as
\begin{footnotesize}
\begin{lstlisting}
sage: for a in S.keys():
        print(r"S^\mu(%s)=%s" % (latex(a), latex(S[a])), "\n")
\end{lstlisting}
\end{footnotesize}

Finally, we compute the Gauduchon curvature:
\begin{footnotesize}
\begin{lstlisting}
sage: Omega = {(x,y): simp_mat(A[y]*A[x]-A[x]*A[y]
                               -(sum([matrix(n,n,[mu[(x,y)].interior_product(c).constant_coefficient()*b
                                                  for b in A[c].list()])
                                      for c in E.gens()]) if not mu[(x,y)]==0 else zero_matrix(n)))
                               for x in E.gens() for y in E.gens()}
\end{lstlisting}
\end{footnotesize}
It suffices to change \lstinline{A} by \lstinline{S} in the code above to compute the Riemannian curvature.
Moreover, the Chern curvature can be computed as:
\begin{footnotesize}
\begin{lstlisting}
sage: OmegaCh = {b: simp_mat(Omega[b], {t:1}) for b in Omega.keys()}
\end{lstlisting}
\end{footnotesize}

The first and the second Gauduchon-Ricci forms can be computed as:
\begin{footnotesize}
\begin{lstlisting}
sage: rho1 = 1/2*simp_mat(matrix(n,n,[sum([Omega[(E.gens()[i],E.gens()[j])][2*k,2*k+1]
                              for k in range(n/2)]) for i in range(n) for j in range(n)])
                              +matrix(n,n,[sum([J(E.gens()[i]).interior_product(u).constant_coefficient()
                              *J(E.gens()[j]).interior_product(v).constant_coefficient()
                              *Omega[(u,v)][2*k,2*k+1] for u in E.gens() for v in E.gens()
                              for k in range(n/2)]) for i in range(n) for j in range(n)]))
sage: rho2 = 1/2*(simp_mat(sum([Omega[(E.gens()[2*j],E.gens()[2*j+1])] for j in range(n/2)])
                           +sum([J(E.gens()[2*j]).interior_product(u).constant_coefficient()
                           *J(E.gens()[2*j+1]).interior_product(v).constant_coefficient()
                           *Omega[(u,v)] for u in E.gens() for v in E.gens() for j in range(n/2)])))
\end{lstlisting}
\end{footnotesize}
In particular, the first and the second Chern-Ricci curvatures are both zero, as shown by computing:
\begin{footnotesize}
\begin{lstlisting}
sage: rho1Ch = simp_mat(rho1, {t:1})
sage: rho2Ch = simp_mat(rho2, {t:1})
\end{lstlisting}
\end{footnotesize}

Finally, the scalar curvature is computed either as:
\begin{footnotesize}
\begin{lstlisting}
sage: scal = 2*sum([rho2[2*j,2*j+1] for j in range(n/2)])
\end{lstlisting}
\end{footnotesize}
or as:
\begin{footnotesize}
\begin{lstlisting}
sage: scal = 2*sum([rho1[2*j,2*j+1] for j in range(n/2)])
\end{lstlisting}
\end{footnotesize}
giving zero. (We stress here that we can use the built-in methods \lstinline{simplify_full} or \lstinline{factor} of \lstinline{sage.symbolic.expression.Expression}.)

\subsection{The Kodaira surface (see Section \ref{sec:kodaira})} \label{app:kodaira} \hfill \par

%The following SageMath worksheet is available at \url{https://cocalc.com/share/public_paths/86f640d5edcdace1d1b35c3ec72d9761bb17fe6d}.
We can perform the computations for the Kodaira surface as in the previous Section, with small changes, starting by setting the dimension:
\begin{footnotesize}
\begin{lstlisting}
sage: n = 4
\end{lstlisting}
\end{footnotesize}

Here the code to construct the differential:
\begin{footnotesize}
\begin{lstlisting}
sage: f3 = -alpha*r*E.gens()[0]+beta*r*E.gens()[1]+v*E.gens()[3]
sage: d = E.coboundary({
            (0,1): -1/r^2*f3,
            (0,2): alpha/(r*v)*f3,
            (0,3): beta/(r*v)*f3,
            (1,2): -beta/(r*v)*f3,
            (1,3): alpha/(r*v)*f3,
            (2,3): -(alpha^2+beta^2)/(v^2)*f3
        })
\end{lstlisting}
\end{footnotesize}

\subsection{The Kodaira-Thurston almost-complex \texorpdfstring{$4$}{4}-manifold (see Section \ref{sec:kodairathurston-almost})}\label{app:kodaira-thurston} \hfill \par

We can perform the computations as in the previous Sections. We present here the code in order to compute the complex structure equations.
%It is available at \url{https://cocalc.com/share/public_paths/158b3870fed3f863eb44df92435da5c354fecfe4}.

We start from the standard real frame $(e_0,e_1,e_2,e_3)$ with structure equations determined by $[e_0,e_1]=-e_3$:
\begin{footnotesize}
\begin{lstlisting}
sage: n = 4
sage: E = ExteriorAlgebra(SR, 'e', n)
sage: d = E.coboundary({
            (0,1): -E.gens()[3],
            (0,2): 0,
            (0,3): 0,
            (1,2): 0,
            (1,3): 0,
            (2,3): 0
        })
sage: print([d(b) for b in E.gens()])
[0, 0, 0, -e0*e1]
\end{lstlisting}
\end{footnotesize}

The non-integrable almost-complex structure is given by $Je_0=e_2$ and $Je_1=e_3$:
\begin{footnotesize}
\begin{lstlisting}
sage: Jmat=block_matrix([ [zero_matrix(2), -identity_matrix(2)],
                         [identity_matrix(2), zero_matrix(2)] ])
sage: J = E.lift_morphism(Jmat)
sage: print([J(b) for b in E.gens()])
[e2, e3, -e0, -e1]
\end{lstlisting}
\end{footnotesize}
We check the non-integrability:
\begin{footnotesize}
\begin{lstlisting}
sage: Nij={(a,b): -J(Lie(J(a),b)+Lie(a,J(b))+Lie(J(a),J(b))-Lie(a,b))
           for a in E.gens() for b in E.gens()}
sage: [Nij[(a,b)] for a in E.gens() for b in E.gens() if Nij[(a,b)]!=0]
[e1, e1, -e1, -e1, e1, -e1, -e1, e1]
\end{lstlisting}
\end{footnotesize}
We construct the coframe of $(1,0)$-forms $\f^1\=e^0-\ti e^2$, $\f^1\=e^1-\ti e^3$, where $(e^0,e^1,e^2,e^3)$ denotes the dual basis of $(e_0,e_1,e_2,e_3)$:
\begin{footnotesize}
\begin{lstlisting}
sage: varphi = [E.gens()[j]-I*J(E.gens()[j]) for j in range(n/2)]
sage: barvarphi = [E.gens()[j]+I*J(E.gens()[j]) for j in range(n/2)]
sage: varphi
[e0 - I*e2, e1 - I*e3]
\end{lstlisting}
\end{footnotesize}
Notice that the convention by SageMath for the action of the complex structure on the dual differs from our notation:
\begin{footnotesize}
\begin{lstlisting}
sage: all([J(b)==I*b for b in varphi])
True
sage: all([J(b)==-I*b for b in barvarphi])
True
\end{lstlisting}
\end{footnotesize}
We check that the structure equations in this coframe are
\begin{footnotesize}
\begin{lstlisting}
sage: [d(b) for b in varphi]
[0, I*e0*e1]
\end{lstlisting}
\end{footnotesize}
namely,
$$ \diff\f^1=0 \,\,,\qquad \diff\f^2=\tfrac{\ti}{4}\big(\f^{1}\wedge\f^2+\f^1\wedge\bar\f^2-\f^2\wedge\bar\f^1+\bar\f^1\wedge\bar\f^2\big) \,\,.$$

The generic almost-Hermitian metric is given by
$$2\omega=-\ti r^2\f^1\wedge\bar\f^1-\ti\sigma^2\f^2\wedge\bar\f^2+u\f^1\wedge\bar\f^2-\bar{u}\f^2\wedge\bar\f^1 \,\,,$$
where $r,\sigma\in\bR^{>0}$ and $u\in\bC$ satisfy $r^2\sigma^2>|u|^2$ (compare with Remark \ref{rmk:std-iwasawa}).
The following code will allow us to derive the $g$-orthonormal frame with respect to the generic metric $g$ associated to $\omega$ and $J$:
\begin{footnotesize}
\begin{lstlisting}
sage: _ = var("r sigma x y")
sage: omega = 1/2*(-I*r^2*varphi[0]*barvarphi[0]-I*sigma^2*varphi[1]*barvarphi[1]
                   +(x+I*y)*varphi[0]*barvarphi[1]-(x-I*y)*varphi[1]*barvarphi[0])
sage: P = matrix(n,n,[omega.interior_product(a*J(b)).constant_coefficient()
                      for a in E.gens() for b in E.gens()])

sage: def scalar_product(a,b,P=P):
        return (a.transpose()*P*b)[0,0]

sage: def GS(e):
        ftmp = []
        for j in range(len(e)):
            ftmp.append(e[j]-sum([scalar_product(e[j],ftmp[k])/scalar_product(ftmp[k],ftmp[k])*ftmp[k]
                                  for k in range(0,j)]))
        f = [1/sqrt(scalar_product(ftmp[j],ftmp[j]))*ftmp[j] for j in range(len(e))]
        return(f)

sage: fmat = GS([identity_matrix(n)[:,j] for j in range(n)])
sage: f = [sum([emat[j][k,0]*E.gens()[k] for k in range(n)]) for j in range(n)]
\end{lstlisting}
\end{footnotesize}
We now make the frame $w$ also $(J,g)$-unitary:
\begin{footnotesize}
\begin{lstlisting}
sage: w = [1/sqrt(2)*(f[j]-J(f[j])) for j in [0,3]]+[1/sqrt(2)*(f[j]+J(f[j])) for j in [0,3]]
\end{lstlisting}
\end{footnotesize}
namely, it is orthonormal and $J$ acts as $J(w_0)=w_2$, $J(w_1)=w_3$:
\begin{footnotesize}
\begin{lstlisting}
sage: all([matrix(n,n,[omega.interior_product(b*J(c)).constant_coefficient().simplify_full()
                       for b in w for c in w])==identity_matrix(n)]
          + [J(w[0])-w[2]==0, J(w[1])-w[3]==0])
True
\end{lstlisting}
\end{footnotesize}

We are now able to compute the structure equations with respect to the $(J,g)$-unitary frame $(w_0,w_1,w_2,w_3)$:
\begin{footnotesize}
\begin{lstlisting}
sage: muw = {(a,b,c): Lie(w[a], w[b]).interior_product(w[c]).constant_coefficient()
             for a in range(n) for b in range(n) for c in range(n)}
\end{lstlisting}
\end{footnotesize}
We can now construct the Lie algebra by using these structure equations:
\begin{footnotesize}
\begin{lstlisting}
sage: reset('E')
sage: E = ExteriorAlgebra(SR, 'e', n)
sage: struct_eq = {(j,k): sum([mue[(j,k,h)]*E.gens()[h] for h in range(n)])
                   for j in range(n) for k in range(n)}
sage: d = E.coboundary({(a,b): sum([muw[(a,b,c)]*E.gens()[c] for c in range(n)])
                        for a in range(n) for b in range(n)})
\end{lstlisting}
\end{footnotesize}
We check that the Jacobi identity is satisfied:
\begin{footnotesize}
\begin{lstlisting}
sage: all([d(d(b))==0 for b in E.gens()])
True
\end{lstlisting}
\end{footnotesize}

We now proceed by constructing the variables \lstinline{mu}, \lstinline{F}, \lstinline{Nij}, \lstinline{S}.
We note that the almost-Hermitian structure is almost-K\"ahler, namely, $\diff\omega=0$.
We need to modify the formula for \lstinline{A}, in order to include the terms coming from the non-vanishing Nijenhuis tensor:
\begin{footnotesize}
\begin{lstlisting}
sage: A = {x: simp_mat(matrix(n,n,[S[x][E.gens().index(y),E.gens().index(z)]
                                   +(t+1)/4*F.interior_product(x*J(y)*J(z)).constant_coefficient()
                                   +(t-1)/4*F.interior_product(x*y*z).constant_coefficient()
                                   +1/4*Nij[(y,z)].interior_product(x).constant_coefficient()
                                   for y in E.gens() for z in E.gens()])) for x in E.gens()}
\end{lstlisting}
\end{footnotesize}
We also compute the variables \lstinline{ACh}, the curvatures \lstinline{Omega} and \lstinline{OmegaCh}, the Ricci forms \lstinline{Ric1} and \lstinline{Ric2}, \lstinline{Ric1Ch} and \lstinline{Ric2Ch}, the scalar curvature \lstinline{scal}:
\begin{footnotesize}
\begin{lstlisting}
sage: latex(scal.simplify_full())
-\frac{r^{2}}{r^{4} \sigma^{4} - 2 \, r^{2} \sigma^{2} x^{2}
    + x^{4} + y^{4} - 2 \, {\left(r^{2} \sigma^{2} - x^{2}\right)} y^{2}}
\end{lstlisting}
\end{footnotesize}


\begin{thebibliography}{20}

\bibitem{angella-calamai-spotti-2}
D. Angella, S. Calamai, C. Spotti,
Remarks on Chern-Einstein Hermitian metrics,
{\em Math. Z.} \textbf{295} (2020), no. 3-4, 1707--1722.

\bibitem{angella-dloussky-tomassini}
D. Angella, G. Dloussky, A. Tomassini,
On Bott-Chern cohomology of compact complex surfaces,
{\em Ann. Mat. Pura Appl. (4)} \textbf{195} (2016), no. 1, 199--217.

\bibitem{angella-pediconi-1}
D. Angella, F. Pediconi,
On Cohomogeneity one Hermitian non-K\"ahler metrics,
{\tt arXiv:2010.08475}.

\bibitem{angella-pediconi-2}
D. Angella, F. Pediconi,
On the linearization stability of the Chern-scalar curvature,
{\tt arXiv:2106.09990}.

\bibitem{arroyo-lafuente}
R. M. Arroyo, R. A. Lafuente,
On the signature of the Ricci curvature on nilmanifolds,
{\tt arXiv:2009.11464}.

\bibitem{barth-peters-hulek-vandeven}
W. P. Barth, K. Hulek, Ch. A. M. Peters, A. Van de Ven,
{\em Compact complex surfaces}, Second edition, Ergebnisse der Mathematik und ihrer Grenzgebiete, Springer-Verlag, Berlin, 2004.

\bibitem{belgun}
F. A. Belgun,
On the metric structure of non-K\"ahler complex surfaces,
{\em Math. Ann.} 317 (2000), no. 1, 1--40.

\bibitem{besse}
A. L. Besse,
{\em Einstein manifolds},
Ergebnisse der Mathematik und ihrer Grenzgebiete (3), \textbf{10}, Springer-Verlag, Berlin, 1987.

\bibitem{bohm-lafuente}
C. B\"ohm, R. Lafuente,
Immortal homogeneous Ricci flows,
{\em Invent. Math.} \textbf{212} (2018), no. 2, 461--529.

\bibitem{bohm-lafuente-simon}
C. B\"ohm, R. Lafuente, M. Simon,
Optimal curvature estimates for homogeneous Ricci flows,
{\em Int. Math. Res. Not. IMRN} \textbf{2019} (2019), 4431--4468.

\bibitem{boling}
J. Boling,
Homogeneous solutions of pluriclosed flow on closed complex surfaces,
{\em J. Geom. Anal.} \textbf{26} (2016), no. 3, 2130--2154.

\bibitem{cheeger-gromov-1}
J. Cheeger, M. L. Gromov,
Collapsing Riemannian manifolds while keeping their curvature bounded I,
{\em J. Differential Geom.} \textbf{23} (1986), 309--346.

\bibitem{cheeger-gromov-2}
J. Cheeger, M. L. Gromov,
Collapsing Riemannian manifolds while keeping their curvature bounded II,
{\em J. Differential Geom.} \textbf{32} (1990), 269--298.

\bibitem{ricci-flow}
B. Chow, S.-C. Chu, D. Glickenstein, C. Guenther, J. Isenberg, T. Ivey, D. Knopf, P. Lu, F. Luo, L. Ni, {\em The Ricci flow: techniques and applications. Part I. Geometric aspects}, Mathematical Surveys and Monographs, \textbf{135}, American Mathematical Society, Providence, RI, 2007.

\bibitem{console-nicolodi}
S. Console, L. Nicolodi,
Infinitesimal characterization of almost Hermitian homogeneous spaces,
{\em Comment. Math. Univ. Carolin.} \textbf{40} (1999), 713--721.

\bibitem{enrietti-fino-vezzoni}
N. Enrietti, A. Fino, L. Vezzoni,
The pluriclosed flow on nilmanifolds and tamed symplectic forms,
{\em J. Geom. Anal.} \textbf{25} (2015), 883--909.

\bibitem{fino-parton-salamon}
A. Fino, M. Parton, S. Salamon,
Families of strong KT structures in six dimensions,
{\em Comment. Math. Helv.} \textbf{79} (2004), no. 2, 317--340.

\bibitem{gauduchon-mathann}
P. Gauduchon,
La $1$-forme de torsion d'une vari\'et\'e hermitienne compacte,
{\em Math. Ann.} \textbf{267} (1984), no. 4, 495--518.

\bibitem{gauduchon-bumi}
P. Gauduchon,
Hermitian connections and Dirac operators,
{\em Boll. Un. Mat. Ital. B (7)} \textbf{11} (1997), no. 2, suppl., 257--288.

\bibitem{gilbarg-trudinger}
D. Gilbarg, N. S. Trudinger,
{\em Elliptic partial differential equations of second order},
Springer-Verlag, Berlin, 2001.

\bibitem{glickenstein}
D. Glickenstein,
Precompactness of solutions to the Ricci flow in the absence of injectivity radius estimates,
{\em Geom. Topol.} \textbf{7} (2003), 487--510.

\bibitem{grauert}
H. Grauert,
On Levi's problem and the imbedding of real-analytic manifolds,
{\em Ann. of Math. (2)} \textbf{68} (1958), 460--472.

\bibitem{hirsch}
M. W. Hirsch,
{\em Differential topology}, Corrected reprint of the 1976 original, Graduate Texts in Mathematics, \textbf{33}, Springer-Verlag, New York, 1994.

\bibitem{kiricenko}
V. F. Kiri\v{c}enko,
On homogeneous Riemannian spaces with an invariant structure tensor,
{\em Dokl. Akad. Nauk SSSR} \textbf{252} (1980), no. 2, 291--293.
English translation: {\em Soviet Math. Dokl.} \textbf{21} (1980), no. 3, 734--737.

\bibitem{kobayashi-nomizu-2}
S. Kobayashi, K. Nomizu,
{\em Foundations of differential geometry. Vol. II},
Reprint of the 1969 original, Wiley Classics Library, A Wiley-Interscience Publication, John Wiley \& Sons, Inc., New York, 1996.

\bibitem{kostant}
B. Kostant,
{\em Holonomy and the Lie algebra of infinitesimal motions of a Riemannian manifold},
{\em Trans. Amer. Math. Soc.} \textbf{80} (1955), 528--542.
 
\bibitem{lafuente}
R. A. Lafuente,
Scalar curvature behavior of homogeneous Ricci flow,
{\em J. Geom. Anal.} \textbf{25} (2015), 2313--2322.

\bibitem{lafuente-pujia-vezzoni}
R. A. Lafuente, M. Pujia, L. Vezzoni,
Hermitian curvature flow on unimodular Lie groups and static invariant metrics,
{\em Trans. Amer. Math. Soc.} \textbf{373} (2020), no. 6, 3967--3993.

\bibitem{lauret-jlms}
J. Lauret,
Convergence of homogeneous manifold,
{\em J. Lond. Math. Soc.} \textbf{86} (2012), 701--727.

\bibitem{lauret-mathz}
J. Lauret,
Ricci flow of homogeneous manifolds,
{\em Math. Z.} \textbf{274} (2013), 373--403.

\bibitem{lauret-tams}
J. Lauret,
Curvature flows for almost-Hermitian Lie groups,
{\em Trans. Amer. Math. Soc.} \textbf{367} (2015), 7453--7480.

\bibitem{lauret-rend}
J. Lauret,
Geometric flows and their solitons on homogeneous spaces,
{\em Rend. Semin. Mat. Univ. Politec. Torino} \textbf{74} (2016), 55--93.

\bibitem{lott}
J. Lott,
On the long-time behavior of type-III Ricci flow solutions,
{\em Math. Ann.} \textbf{339} (2007), 627--666.

\bibitem{meusers}
C. Meusers,
High Singer invariant and equality of curvature,
{\em Bull. Belg. Math. Soc. Simon Stevin} \textbf{9} (2002), 491--502.

\bibitem{milnor}
J. Milnor,
Curvatures of left invariant metrics on Lie groups,
{\em Adv. Math.} \textbf{21} (1976), no. 3, 293--329.

\bibitem{morrey}
C. B. Morrey,
The analytic embedding of abstract real-analytic manifolds,
{\em Ann. of Math. (2)} \textbf{68} (1958), 159--201.

\bibitem{nicolodi-tricerri}
L. Nicolodi, F. Tricerri,
On two theorems of I. M. Singer about homogeneous spaces,
{\em Ann. Global Anal. Geom.} \textbf{8} (1990), no. 2, 193--209.

\bibitem{nomizu}
K. Nomizu,
On local and global existence of Killing vector fields,
{\em Ann. of Math. (2)} \textbf{72} (1960), 105--120.

\bibitem{palais}
R. S. Palais,
A global formulation of the Lie theory of transformation groups,
{\em Mem. Amer. Math. Soc.} \textbf{22} (1957).

\bibitem{panelli-podesta}
F. Panelli, F. Podest\`a,
Hermitian Curvature Flow on compact homogeneous spaces,
{\tt arXiv:1903.10273}.

\bibitem{pediconi-thesis}
F. Pediconi,
{\em Geometric aspects of locally homogeneous Riemannian spaces},
PhD thesis, Universit\`a di Firenze, \url{http://hdl.handle.net/2158/1197175}.

\bibitem{pediconi-blms}
F. Pediconi,
A local version of the Myers-Steenrod theorem,
{\em Bull. Lond. Math. Soc.} \textbf{52} (2020), no. 5, 871--884.

\bibitem{pediconi-geomded}
F. Pediconi,
Convergence of locally homogeneous spaces,
{\em Geom. Dedicata} \textbf{211} (2021), 105--127.

\bibitem{pediconi-annsns}
F. Pediconi,
A compactness theorem for locally homogeneous spaces,
to appear in {\em Ann. Sc. Norm. Super. Pisa Cl. Sci.}.

\bibitem{pediconi-pujia}
F. Pediconi, M. Pujia,
Hermitian curvature flow on complex locally homogeneous surfaces,
{\em Ann. Mat. Pura Appl. (4)} \textbf{200} (2021), no. 2, 815--844.

\bibitem{petersen}
P. Petersen,
{\em Riemannian geometry}, Third edition, Graduate Texts in Mathematics, \textbf{171}, Springer, Cham, 2016. 

\bibitem{podesta}
F. Podest\`a,
Homogeneous Hermitian manifolds and special metrics,
{\em Transform. Groups} \textbf{23} (2018), 1129--1147.

\bibitem{sagemath}
\emph{{S}ageMath, the {S}age {M}athematics {S}oftware {S}ystem ({V}ersion 9.3)}, The Sage Developers, 2021, \url{https://www.sagemath.org}.

\bibitem{sekigawa}
K. Sekigawa,
Notes on homogeneous almost Hermitian manifolds,
{\em Hokkaido Math. J.} \textbf{7} (1978), no. 2, 206--213.

\bibitem{spiro}
A. Spiro,
Lie pseudogroups and locally homogeneous Riemannian spaces,
{\em Boll. Un. Mat. Ital. B (7)} \textbf{6} (1992), no. 4, 843--872.

\bibitem{thurston}
W. P. Thurston,
Some simple examples of symplectic manifolds,
{\em Proc. Am. Math. Soc.} \textbf{55} (1976). 467--468.

\bibitem{tosatti-weinkove}
V. Tosatti, B. Weinkove,
The Chern-Ricci flow on complex surfaces,
{\em Compos. Math.} \textbf{149} (2013), no. 12, 2101--2138.

\bibitem{tralle-oprea}
A. Tralle, J. Oprea,
{\em Symplectic manifolds with no K\"ahler structure},
Lecture Notes in Mathematics, \textbf{1661}, Springer-Verlag, Berlin, 1997.

\bibitem{tricerri}
F. Tricerri,
Locally homogeneous Riemannian manifolds,
{\em Differential geometry (Turin, 1992)},
{\em Rend. Sem. Mat. Univ. Politec. Torino} \textbf{50} (1992), no. 4, 411--426 (1993).

\bibitem{tricerri-vanhecke}
F. Tricerri, L. Vanhecke, {\em Homogeneous structures on Riemannian manifolds}, London Mathematical Society Lecture Note Series, \textbf{83}, Cambridge University Press, Cambridge, 1983.

\bibitem{ugarte-villacampa}
L. Ugarte, R. Villacampa,
Balanced Hermitian geometry on $6$-dimensional nilmanifolds,
{\em Forum Math.} \textbf{27} (2015), no. 2, 1025--1070.

\bibitem{ustinovskij}
Y. Ustinovskiy,
Hermitian curvature flow on complex homogeneous manifolds,
{\em Ann. Sc. Norm. Super. Pisa Cl. Sci. (5)} \textbf{21} (2020), 1553--1572.

\end{thebibliography}
\end{document}